\documentclass[12pt]{amsart}
\usepackage{amsmath}
\usepackage{amsthm}
\usepackage{bbm}
\usepackage{amsfonts,amssymb,bm}
\usepackage{fancyhdr}
\usepackage{mathrsfs}
\usepackage{appendix}
\usepackage{graphicx}
\usepackage{comment}
\usepackage[all]{xy}
\usepackage{color}
\usepackage{enumerate}
\usepackage{tikz-cd}
\usepackage{mathtools}

\usepackage{float} 
\usepackage{amsthm}

\newtheorem{thm}{Theorem}[section]

\newtheorem{lem}[thm]{Lemma}
\newtheorem{cor}[thm]{Corollary}
\newtheorem{prop}[thm]{Proposition}
\newtheorem*{claim}{Claim}

\theoremstyle{definition}
\newtheorem{defi}[thm]{Definition}
\newtheorem{rmk}[thm]{Remark}
\newtheorem{construction}[thm]{Construction}

\usepackage{hyperref}
\usepackage{times} 

\def\Aut{{\rm Aut}}

\def\Supp{{\rm Supp}}

\def\Gr{{\rm Gr}}

\def\wt{{\rm wt}}

\def\Proj{{\rm Proj}}
\def\Spec{{\rm Spec}}

\def\mult{{\rm mult}}

\def\Image{{\rm Image}}

\def\div{{\rm div}}

\def\codim{{\rm codim}}

\def\B0{\mathbf{0}}

\newcommand{\git}{\mathord{/\mkern-6mu/}} 

\newcommand{\IA}{{\mathbb A}}

\newcommand{\IG}{{\mathbb G}}

\newcommand{\Ik}{{\mathbbm k}}

\newcommand{\IN}{{\mathbb N}}
 
\newcommand{\IP}{{\mathbb P}} 
\newcommand{\IQ}{{\mathbb Q}}

\newcommand{\IV}{{\mathbb V}}

\newcommand{\IZ}{{\mathbb Z}}

\newcommand{\CB}{{\mathcal B}}
\newcommand{\CC}{{\mathcal C}}
\newcommand{\CD}{{\mathcal D}}

\newcommand{\CH}{{\mathcal H}}

\newcommand{\CL}{{\mathcal L}}
\newcommand{\CM}{{\mathcal M}} 

\newcommand{\CO}{{\mathcal O}} 
\newcommand{\CP}{{\mathcal P}}

\newcommand{\CS}{{\mathcal S}}

\newcommand{\CV}{{\mathcal V}}

\newcommand{\CW}{{\mathcal W}}
\newcommand{\CX}{{\mathcal X}}
\newcommand{\CY}{{\mathcal Y}}

\newcommand{\cy}{{\rm{CY}}}
\newcommand{\ksba}{{\rm{KSBA}}}

\newcommand{\fm}{\mathfrak{m}}

\newcommand{\fh}{\mathfrak{h}}
\newcommand{\fg}{\mathfrak{g}}

\newcommand{\la}{\langle}
\newcommand{\ra}{\rangle}

\newcommand{\lam}{\lambda}
\newcommand{\D}{\Delta}

\newcommand{\vep}{\varepsilon}
\newcommand{\Hilb}{\mathrm{Hilb}}

\newcommand{\leftsquigarrow}{%
  \mathrel{\reflectbox{$\rightsquigarrow$}}%
}

\def\oF{{\overline{F}}}
\def\oR{{\overline{R}}}

\title{log Calabi-Yau wall crossing for moduli of cubic surfaces}
\author{Hanfeng Chu}
\address{Shanghai Center for Mathematical Sciences, Fudan University, Shanghai, 200438, China}
\curraddr{}
\email{hfchu23@m.fudan.edu.cn}
\thanks{}
\keywords{}
\date{}
\dedicatory{}

\begin{document}
\begin{abstract}
    We study the wall crossing for the moduli space of unmarked cubic surfaces at the log Calabi-Yau threshold. We prove the existence of a good moduli space for the corresponding moduli stack of boundary polarized Calabi-Yau pairs with Gorenstein underlying surfaces and describe the wall crossing morphisms explicitly. More precisely, the morphism from the K-moduli space to the log Calabi-Yau moduli space is an isomorphism, whereas the morphism from the KSBA moduli space contracts a divisor isomorphic to $\IP^3$ to a point and is an isomorphism away from this divisor.
\end{abstract}
\maketitle
\section{Introduction}
    The development of several approaches to moduli theory has led to substantial progress in the study of compactifications of the moduli space of smooth cubic surfaces. The twenty-seven lines on a smooth cubic surface provide a canonical choice of the boundary divisor, allowing one to  construct compactifications using K-stability and KSBA theory. We refer to \cite{Xubook} and \cite{Kollar23} for general theories.
 
     More precisely, let $X$ be a smooth cubic surface and $D$ be the sum of its twenty-seven lines. For a rational number $c\in [0,1]$, the pair $(X,cD)$ is log Fano if $c<1/9$ and of log general type if $c>1/9$. Thus K-stability and KSBA theory apply on the two sides of the threshold $c=1/9$, respectively.

    The moduli theory on either side of the log Calabi-Yau threshold  is already understood. On the log Fano side, the K-moduli spaces of unmarked cubic surfaces have no walls for $0\leq c<1/9$ and are all isomorphic to the GIT moduli space of cubic surfaces \cite[Section 3.1]{Zhao24}. On the KSBA side, previous work has  mainly focused on marked cubic surfaces. The moduli space of smooth marked cubic surfaces admits the classical cross-ratio compactification constructed by Naruki \cite{Naruki82}. Here, the marked moduli problem keeps track of a labeling of the 27 lines, whereas the unmarked problem remembers only the sum of lines. The latter is obtained from the former by taking the quotient by the natural $W(E_6)$-action permuting the markings. At weight $(1+\vep)/9$ for $0<\vep\ll1$, the normalization of the marked KSBA compactification was identified with Naruki's cross-ratio compactification \cite{GKS21}. This normalization was subsequently shown to be unnecessary: the marked KSBA compactification itself is smooth and fine \cite{FSW25}. The wall crossings for larger weights in the marked setting were determined in \cite{Schock24}.
    
    A natural problem is therefore to understand the wall crossing at the log Calabi-Yau threshold. The theory of boundary polarized CY pairs developed in \cite{ABB23} and \cite{BL24} provides the moduli space at $c=1/9$. In \cite[Section 7.1]{BL24}, Blum-Liu constructed the CY-moduli stack of unmarked cubic surfaces $\CY^{\cy}$ and defined the K-moduli substack $\CY^{\mathrm{K}}$ (resp. KSBA substack $\CY^{\ksba}$) as the open substack of $\CY^{\cy}$ parametrizing boundary polarized CY pairs $(X,D)$ such that $(X,(1-\vep)D)$ is K-semistable (resp. $(X,(1+\vep)D)$ is KSBA-stable) for $0<\vep\ll1$. Then they constructed a wall crossing diagram of birational morphisms between the corresponding (asymptotically) good moduli spaces \cite[Theorem 7.7]{BL24}.

    In this paper, we consider the index one substack $\CY_1^{\cy}\subset\CY^{\cy}$ parametrizing boundary polarized CY pairs $(X,D)$ such that $K_X$ is Cartier. See Construction \ref{Construction. moduli Ycy} for details. Our first result establishes the existence of a good moduli space for $\CY_1^{\cy}$.
    \begin{thm}[Theorem {\ref{Theorem. Existence of good moduli}}]
    \label{Theorem. 1.1}
        The index one CY-moduli stack of unmarked cubic surfaces $\CY_1^{\cy}$ admits a good moduli space morphism
        \[
            \phi_{\cy}:\CY_1^{\cy}\longrightarrow Y_1^{\cy},
        \]
        where $Y_1^{\cy}$ is a separated algebraic space of finite type. 
    \end{thm}
    The substacks $\CY^{\mathrm{K}}$ and $\CY^{\ksba}$ are both contained in $\CY_1^{\cy}$ and admit good moduli spaces $Y^{\mathrm{K}}$ and $Y^{\ksba}$ respectively (Proposition \ref{Proposition. good moduli of K and KSBA}). The wall crossing  is concentrated at the point $p_0=[(X_0,\frac19D_0)]$ represented by 
    \[
        X_0=(x_1x_2x_3=0)\subset\IP^3_{[x_0:x_1:x_2:x_3]},\quad D_0=(x_0^9=0)|_{X_0}.
    \]
     The point $p_0$ is closed in $\CY_1^{\cy}$ and belongs to neither $\CY^{\mathrm{K}}$ nor $\CY^{\ksba}$.
     \begin{thm}[Corollary {\ref{Corollary. describe wall crossing morphisms}}]
     \label{Theorem. 1.2}
         The inclusions of $\CY^{\mathrm{K}}$ and $\CY^{\ksba}$ in $\CY_1^{\cy}$ induce morphisms of good moduli spaces
         \[
         Y^{\mathrm{K}}\xrightarrow{\ \pi_{\mathrm{K}}\ \ }Y_1^{\cy}\xleftarrow{\pi_{\ksba}}Y^{\ksba}
         \]
    with the following properties:
    \begin{enumerate}
        \item The morphism $\pi_{\mathrm{K}}$ is an isomorphism. In particular, $Y_1^{\cy}$ is isomorphic to the GIT moduli space of cubic surfaces.
        \item The morphism $\pi_{\ksba}$ is an isomorphism over $Y_1^{\cy}\backslash\{q_0\}$, where $q_0=\phi_{\cy}(p_0)$, and contracts the divisor
        \[
            E:=(\pi_{\ksba}^{-1}(q_0))_{\mathrm{red}}\cong\IP^3
        \]
        to $q_0$.
    \end{enumerate}
    \end{thm}
The paper is organized as follows. In Section \ref{Section. pre}, we recall preliminaries on boundary polarized CY pairs, good moduli spaces and local VGIT. In Section \ref{Section. moduli space}, we construct the index one CY-moduli stack of unmarked cubic surfaces and prove the existence of its good moduli space. In Section \ref{Section. sec 4}, we construct an explicit family in the moduli stack and use it to obtain a local quotient presentation at $p_0$. In Section \ref{Section. sec 5}, we establish the local VGIT picture at $p_0$ and use it to describe the wall crossing morphisms explicitly.
\subsection*{Acknowledgements}The author is grateful to Yuchen Liu for suggesting this question and for many helpful conversations. He  thanks his advisor, Zhiyuan Li, for his constant support and many valuable comments on this work. He also thanks Long Pan, Fei Si, Haoyu Wu, and Junyan Zhao for fruitful discussions. The author is supported by NSFC Grant No. 12121001.
\subsection*{AI disclosure}
LLMs were used in the construction of (\ref{Equation. Family on U}), Section \ref{Section. sec 4}, and Section \ref{Section. sec 5}. Actually, the main problem studied in this paper was suggested to the author by Yuchen Liu during a discussion in July 2025. In September 2025, the author proved the existence of a good moduli space (Theorem \ref{Theorem. 1.1}) and subsequently attempted to describe the wall crossing morphisms by establishing a local VGIT picture. The author initially attempted to follow the strategy of \cite{ADL24}, but this approach was unsuccessful.

In December 2025, during Junyan Zhao's visit to SCMS, he suggested constructing an explicit family. Following several fruitful discussions, the author obtained some rough ideas of the explicit construction. After sharing ideas with LLMs, the models proposed the family in (\ref{Equation. Family on U}), which closely aligned with the author's intended construction. The tool was also used to assist in the proof of Proposition \ref{Proposition. Etaleness of Phi}, Lemma \ref{Lemma. morphism of quotient of VGIT}, and Lemma \ref{Lemma. shrinking U A1A2}. In addition, LLMs were used for language editing, assistance with literature searches, and discussion of certain mathematical arguments.

All mathematical statements, constructions, and proofs appearing in this manuscript were independently checked by the author. The author assumes full responsibility for the manuscript.
\section{Preliminaries}
\label{Section. pre}
Throughout, we work over an algebraically closed field $\Ik$ of characteristic 0.
\subsection{Moduli stack of boundary polarized CY pairs}
\begin{defi}
    A projective slc pair $(X,D)$ is called a \emph{Calabi-Yau pair}, or \emph{CY pair}, if $K_X+D\sim_{\IQ}0$. A CY pair $(X,D)$ is called \emph{boundary polarized} if $D$ is an effective ample $\IQ$-Cartier divisor.
\end{defi}
\begin{rmk}
    For such a pair, $-K_X$ is ample and $\IQ$-Cartier. The more general convention in \cite{BL24} allows a boundary divisor $B$ and writes a pair as $(X,B+D)$, with $(X,B)$ an slc log Fano pair in the sense of \cite[Definition 2.1]{ABB23}. We use only the case $B=0$.
\end{rmk}
\begin{defi}
    Let $f:X\to T$ be a flat finite type morphism with fibers of pure dimension $n$. A subscheme $D\subset X$ is a \emph{relative Mumford divisor} if there is an open set $U\subset X$ such that
    \begin{enumerate}
        \item $\codim_{X_t}(X_t\backslash U_t)\geq 2$ for any $t\in T$,
        \item $D|_U$ is a relative Cartier divisor (i.e. $D|_U$ is a Cartier divisor on $U$ and does not contain irreducible components of fibers),
        \item $D$ is the closure of $D|_U$, and
        \item $X_t$ is smooth at the generic points of $D_t$ for any $t\in T$.
    \end{enumerate}
When $T=\Spec(\Ik)$, we simply call $D$ a \emph{Mumford divisor}.
\end{defi}

If $D\subset X$ is a relative Mumford divisor of $f:X\to T$ and $T'\to T$ is a morphism, then its \emph{divisorial pullback} $D_{T'}$ on $X_{T'}:=X\times_TT'$ is defined as the closure of the pullback of $D|_U$ to $U':=U\times_TT'$. Note that $D_t$ in (4) denotes the divisorial pullback and does not always agree with the scheme-theoretic fiber. 
\begin{defi}[{\cite[Definition 2.27]{BL24}}]
\label{Definition. family of bpcy}
    Fix a positive rational number $c$ and a positive integer $N$ with $Nc\in\IZ$. A \emph{family of boundary polarized CY pairs with coefficient $c$ and index dividing $N$} over a Noetherian scheme $T$ consists of a pair $(X,D)$ over $T$ satisfying the following conditions:
    \begin{enumerate}
        \item $X\to T$ is a flat projective morphism of schemes;
        \item $D=cD_1$, where $D_1$ is a relative K-flat Mumford divisor in the sense of \cite[Section 7.1]{Kollar23};
        \item $(X_t,D_t)$ is a boundary polarized CY pair for every $t\in T$;
        \item $\omega_{X/T}^{[N]}(ND)\cong_T\CO_X$;
        \item $\omega_{X/T}^{[m]}(qD_1)$ is flat over $T$ and commutes with base change for every $m,q\in\IZ$.
    \end{enumerate}
\end{defi}
\begin{rmk}
    The definition extends to arbitrary bases by the bootstrapping from the Noetherian case, as in \cite[Definition 3.4]{ABB23} and \cite[Remark 2.30]{BL24}. This is needed to define the moduli stack below.
\end{rmk}
\begin{defi}[{\cite[Definition 2.31]{BL24}}]
    Fix a function $\chi:\IN\to\IZ$. Let $\CM(\chi,N,c)$ denote the category fibered in groupoids over $\text{Sch}_{\Ik}$ where:
    \begin{itemize}
        \item The objects are families of boundary polarized CY pairs $(X,D)\to T$ with coefficient $c$, index dividing $N$, and $\chi(X_t,\omega_{X_t}^{[-m]})=\chi(m)$ for each $t\in T$ and $m\in\IN$.
        \item The morphisms $[(X',D')\to T']\to[(X,D)\to T]$ consist of morphisms of schemes $X'\to X$ and $T'\to T$ such that $X'\to X\times_TT'$ is an isomorphism and $D_1'$ is the divisorial pullback \textup{(}\cite[Definition 4.6]{Kollar23}\textup{)} of $D_1$.
    \end{itemize}
\end{defi}
The following basic property of the moduli stack is proved in \cite[Theorem 2.32]{BL24}.
\begin{thm}
    $\CM(\chi,N,c)$ is an algebraic stack locally of finite type over $\Ik$ with affine diagonal.
\end{thm}
\subsection{Good moduli spaces}
We recall some basic results about good moduli spaces.
\begin{defi}[{\cite{Alper13}}]
    A \emph{good moduli space} $\phi:\CM\to M$ is a quasi-compact morphism from an algebraic stack to an algebraic space such that
    \begin{enumerate}
        \item $\phi_*$ is exact on quasi-coherent sheaves and
        \item the natural morphism $\CO_M\to\phi_*\CO_{\CM}$ is an isomorphism.
    \end{enumerate}
\end{defi}
The following universal property is frequently used.
\begin{prop}[{\cite{Alper13}}]
    If $\phi:\CM\to M$ is a good moduli space and $\CM$ is locally Noetherian, then $\phi$ is universal among maps from $\CM$ to algebraic spaces. 
\end{prop}
In order to prove the existence of a good moduli space, we will use the following criterion.
\begin{thm}[{\cite[Theorem A]{AHLH23}}]
\label{Thereom. S Theta criterion}
    Let $\CM$ be a finite type algebraic stack with affine diagonal. Then the following are equivalent:
    \begin{enumerate}
        \item There exists a good moduli space morphism $\CM\to M$ to a separated algebraic space.
        \item $\CM$ is $S$-complete and $\Theta$-reductive with respect to DVRs essentially of finite type over $\Ik$.
    \end{enumerate}
\end{thm}
We now define terminology in Theorem \ref{Thereom. S Theta criterion}. Let $R$ be a DVR over $\Ik$ with uniformizer $\pi$. Set
\[
    \overline{\mathrm{ST}}_R:=[\Spec(R[s,t]/(st-\pi))/\IG_m]\quad\text{and}\quad\Theta_R:=[\Spec R[t]/\IG_m],
\]
where the $\IG_m$-action is trivial on $R$ and has weights 1 and -1 on $s$ and $t$, respectively. By abuse of notation, we may write $0\in\overline{\mathrm{ST}}_R$ and $0\in\Theta_R$ for the unique closed points.
\begin{defi}
    Let $\CM$ be an algebraic stack.
    \begin{enumerate}
        \item The stack $\CM$ is \emph{S-complete} if every morphism $\overline{\mathrm{ST}}_R\backslash\{0\}\to\CM$ extends uniquely to a morphism $\overline{\mathrm{ST}}_R\to\CM$, for every DVR $R$.
        \item The stack $\CM$ is $\Theta$-\emph{reductive} if every morphism $\Theta_R\backslash\{0\}\to\CM$ extends uniquely to a morphism $\Theta_R\to\CM$, for every DVR $R$.
    \end{enumerate}
    We say $\CM$ is S-complete or $\Theta$-reductive with respect to DVRs essentially of finite type over $\Ik$ if the respective statement in (1) or (2) holds for DVRs essentially of finite type over $\Ik$.
\end{defi}
For the moduli stack of boundary polarized CY pairs, the following result holds.
\begin{thm}[{\cite[Theorem 2.34]{BL24}}]
\label{Theorem. bpCY stack}
    $\CM(\chi,N,c)$ is S-complete, $\Theta$-reductive, and satisfies the existence part of the valuative criterion for properness, all with respect to DVRs essentially of finite type over $\Ik$. 
\end{thm}
\subsection{Source}
We recall the definition of the source and the type of an slc CY pair.
\begin{defi}
    A \emph{dlt modification} of an lc pair $(X,D)$ is a proper birational morphism
    \[
        f:(Y,D_Y)\longrightarrow(X,D),
    \]
    such that $(Y,D_Y)$ is dlt and $K_Y+D_Y=f^*(K_X+D)$. Such modifications exist by \cite[Theorem 1.34]{Kollar13}. If $S\subset Y$ is an lc center of $(Y,D_Y)$, then $(S,D_S)$ is a dlt pair, where $D_S:=\mathrm{Diff}_S^*(D_Y)$ is the different. See \cite[Section 4.2]{Kollar13} for details. 
\end{defi}
\begin{defi}[Source]
    Let $(X,D)$ be a CY pair.
    \begin{enumerate}
        \item If $(X,D)$ is klt, then its source is $\mathrm{Src}(X,D):=(X,D)$.
        \item If $(X,D)$ is not klt, then let $(\overline{X},\overline{D}):=\sqcup_{i=1}^r(\overline{X}_i,\overline{D}_i)$ denote its normalization and its decomposition into irreducible lc pairs. For arbitrary $1\leq i\leq r$, let $(Y,D_Y)\to (\overline{X}_i,\overline{D}_i)$ be a dlt modification and $S\subset Y$ be a minimal lc center. We set $\mathrm{Src}(X,D)$ equal to the crepant birational equivalence class of $(S,D_S)$.
    \end{enumerate}
\end{defi}
The source of a CY pair is independent of the choice of dlt modification and minimal lc center by \cite[Lemma 8.2]{ABB23}, which follows from \cite{Kollar16}. For a surface pair, we record the dimension of its source with the following terminology.
\begin{defi}
    A CY surface pair $(X,D)$ is called
    \begin{enumerate}
        \item Type \uppercase\expandafter{\romannumeral1} if $(X,D)$ is klt,
        \item Type \uppercase\expandafter{\romannumeral2} if $\mathrm{Src}(X,D)$ is a curve pair, and
        \item Type \uppercase\expandafter{\romannumeral3} if $\mathrm{Src}(X,D)$ is a point.
    \end{enumerate}
\end{defi}
\subsection{Test configuration}
\begin{defi}
    A \emph{test configuration} of a boundary polarized CY pair $(X,D)$ is the data of
    \begin{enumerate}
        \item a family of boundary polarized CY pairs $(\CX,\CD)\to\IA^1$,
        \item a $\IG_m$-action on $(\CX,\CD)$ extending the standard action on $\IA^1$, and
        \item an isomorphism $(\CX_1,\CD_1)\cong(X,D)$.
    \end{enumerate}
\end{defi}
The central fiber is preserved by the $\IG_m$-action, and its boundary is $\IG_m$-invariant. Thus the $\IG_m$-action restricts to a $\IG_m$-action on the central pair.
\begin{defi}
    We say that there exists a \emph{weakly special degeneration}
    \[
        (X,D)\rightsquigarrow(X_0,D_0)
    \]
    if there exists a test configuration $(\CX,\CD)$ of $(X,D)$ such that $(\CX_0,\CD_0)\cong(X_0,D_0)$.
\end{defi}
\subsection{S-equivalence}
The following equivalence relation is useful for describing the points of good moduli spaces.
\begin{defi}
    Two boundary polarized CY pairs $(X,D)$ and $(X',D')$ are \emph{S-equivalent} if there exist weakly special degenerations of the pairs
    \[
        (X,D)\rightsquigarrow(X_0,D_0)\leftsquigarrow(X',D')
    \]
to a common boundary polarized CY pair $(X_0,D_0)$.
\end{defi}
\begin{rmk}
\label{Remark. property of S-equiv}
    We use the following properties of S-equivalence.
    \begin{enumerate}
        \item S-equivalence is an equivalence relation \cite[Proposition 6.9]{ABB23}.
        \item If two boundary polarized CY pairs $(X,D)$ and $(X',D')$ are S-equivalent, then their sources are equal \cite[Proposition 8.7]{ABB23}. In particular, they are of the same type.
        \item Let $(X,D)\to C$ and $(X',D')\to C$ be families of boundary polarized CY pairs over the germ of a smooth pointed curve $0\in C$. If there is an isomorphism 
        \[
            (X,D)|_{C^{\circ}}\cong(X',D')|_{C^{\circ}}
        \]
        over $C^{\circ}:=C\backslash0$, then $(X_0,D_0)$ and $(X_0',D_0')$ are S-equivalent \cite[Corollary 1.2]{ABB23}.
    \end{enumerate}
\end{rmk}
\subsection{Local VGIT}\label{Section: local VGIT}
We adopt the terminology of \cite{AFS17}.
\begin{defi}
    \label{Definition. Local presentation}
    Let $\CX$ be an algebraic stack of finite type over an algebraically closed field $k$, and let $x\in\CX(k)$ be a closed point. We say that $f:\CW\to\CX$ is a \emph{local quotient presentation} around $x$ if
    \begin{enumerate}
        \item the stabilizer $G_x$ of $x$ is linearly reductive;
        \item there is an isomorphism $\CW\cong[\Spec A/G_x]$, where $A$ is a finite type $k$-algebra;
        \item the morphism $f$ is \'{e}tale and affine;
        \item there exists a point $w\in\CW(k)$ such that $f(w)=x$ and $f$ induces an isomorphism $G_w\cong G_x$.
    \end{enumerate}
\end{defi}
Let $G$ be a linearly reductive group acting on an affine scheme $X=\Spec A$ by $\sigma:G\times X\to X$. Denote $\IG_m=\Spec(k[t,t^{-1}])$ and let $\theta:G\to\IG_m$ be a character. Set 
\begin{equation}
\label{Equation. def of An}
    A_n:=\{f\in A|\sigma^*(f)=\theta^*(t)^{-n}f\}.
\end{equation}
We define the \emph{VGIT ideals with respect to $\theta$} to be:
\[
    I_{\theta}^{+}=(f\in A|f\in A_n\text{ for some }n>0),
\]
\[
    I_{\theta}^-=(f\in A|f\in A_n\text{ for some }n<0).
\]
The \emph{VGIT $(+)$-chamber and $(-)$-chamber of $X$ with respect to $\theta$} are the open subschemes
\[
    X_{\theta}^+:=X\backslash \IV(I_{\theta}^+)\hookrightarrow X,\quad X_{\theta}^-:=X\backslash \IV(I_{\theta}^-)\hookrightarrow X.
\]
Since the open subschemes $X_{\theta}^+$ and $X_{\theta}^-$ are both $G$-invariant, we have stack-theoretic open immersions
\[
    [X_{\theta}^+/G]\hookrightarrow[X/G]\hookleftarrow[X_{\theta}^-/G].
\]
We refer to these open immersions as the \emph{VGIT $(+)/(-)$-chambers of $[X/G]$ with respect to $\theta$}.
\section{The CY-moduli of unmarked cubic surfaces}
\label{Section. moduli space}
We begin with the construction in \cite[Section 7.1]{BL24}.
\begin{construction}
\label{Construction. moduli Ycy}
For a smooth cubic surface $X$, put
\[
    \chi(m):=\chi(X,\omega_X^{[-m]})=\frac{3}{2}(m^2+m)+1,\quad \CM:=\CM(\chi,9,\frac19).
\]
Let $\CD\CP_{3,9}\subset\CM$ be the open substack parametrizing pairs $(X,\frac19D)$, where $X$ is a smooth cubic surface and $D\in|-9K_X|$ is a smooth curve. Denote the seminormalization of its closure in $\CM$ by $\CD\CP_{3,9}^{\cy}$. Let $\CY$ denote the seminormalization of the locally closed substack of $\CD\CP_{3,9}^{\text{CY}}$ parametrizing pairs $(X,\frac{1}{9}\sum_{i}L_i)$, where $X$ is a smooth cubic surface and $L_i$'s are the 27 lines on it. Define $\CY^{\text{CY}}$ to be the seminormalization of the closure of the image of $\CY$ in $\CD\CP_{3,9}^{\text{CY}}$. Finally, let $\CY_1^{\cy}$ be the substack of $\CY^{\cy}$ parametrizing pairs $(X,\frac19 D)$ such that $K_X$ is Cartier.
\end{construction}
The construction gives morphisms
\begin{equation}
\label{Equation. construct Y1cy}
    \CY_1^{\cy}\hookrightarrow\CY^{\cy}\longrightarrow\CD\CP_{3,9}^{\cy}\longrightarrow\CM,
\end{equation}
where the first morphism is an open immersion and the other two morphisms are both constructed by taking closed substacks and seminormalization respectively. We refer to $\CY_{1}^{\text{CY}}$ as the \emph{index one CY-moduli stack of unmarked cubic surfaces}. 

Define $\CY^{\rm{K}}$ (resp. $\CY^{\text{KSBA}}$) to be the substack of $\CY^{\text{CY}}$ parametrizing boundary polarized CY pairs $(X,\frac19 D)$ such that $(X,\frac{1-\vep}{9}D)$ is K-semistable (resp. $(X,\frac{1+\vep}{9}D)$ is KSBA-stable) for $0<\vep\ll 1$. From \cite[Theorem 3.1 and Corollary 3.8]{Zhao24} and \cite[Theorem 1.5]{GKS21}, we have
\[
    \CY^{\mathrm{K}}\subset \CY_1^{\cy}\quad\text{and}\quad \CY^{\ksba}\subset\CY_1^{\cy}.
\]
\subsection{Existence of the good moduli space} The general member of $\CY_1^{\cy}$ has index 9, so \cite[Theorem 3.6]{BL24} does not directly give the good moduli space needed here. We instead verify the valuative conditions of Theorem \ref{Thereom. S Theta criterion}.

Consider the pair $(X_0,\frac19 D_0)$, where
\begin{equation}
\label{Equation. def of p0}
    X_0=(x_1x_2x_3=0)\subset \IP^3_{[x_0:x_1:x_2:x_3]},\quad D_0=(x_0^9=0)|_{X_0}\in |-9K_{X_0}|.
\end{equation}
The normalization of $(X_0,\frac19D_0)$ is a disjoint union of three copies of $(\IP^2_{[a:b:c]},(abc=0))$, thus $(X_0,\frac19D_0)$ is slc. Let $H_0=\IV(x_0)|_{X_0}$. By adjunction, $K_{X_0}\sim -H_0$, while $D_0=9H_0$. Thus 
 \[
     K_{X_0}+\frac19D_0\sim0,
 \]
 and $\frac19D_0=H_0$ is an ample Cartier divisor. Hence $(X_0,\frac19D_0)$ is a boundary polarized CY pair.
\begin{prop}
    \label{Proposition. K S-equiv CY}
    Any pair in $\CY^{\cy}(\Ik)$ is S-equivalent to either
    \begin{enumerate}
        \item $(X,\frac19\sum_iL_i)$ where $X$ is a cubic surface with at worst $A_1$-singularities and $L_i$'s are the 27 lines on it; or
        \item the pair $(X_0,\frac19D_0)$ in (\ref{Equation. def of p0}).
    \end{enumerate}
\end{prop}
\begin{proof}
We use the same strategy as \cite[Lemma 9.10]{ABB23}. Fix a pair $(X,\frac19D)\in\CY^{\cy}(\Ik)$. Let $(\CX,\frac19\CD)\to T$ be a family of boundary polarized CY pairs over the germ of a curve $0\in T$ such that $(\CX_0,\frac19\CD_0)\cong (X,\frac19 D)$ and $(\CX_K,\frac19D_K)\in \CY(K)$, where $K=K(T)$. Then \cite{OSS16} and \cite[Corollary 3.8]{Zhao24} imply that $(\CX_K,\frac{1-\vep}{9}\CD_K)$ is a K-semistable log Fano pair for $0<\vep\ll1$. By \cite[Theorem 1.3]{LXZ22} and \cite[Theorem 1.1]{BLXZ25}, after a possible finite base change $T'\to T$, there exists a family $(\CX',\frac19\CD')\to T'$ in $\CY^{\mathrm{K}}$ such that $$(\CX'_{K'},\frac19\CD'_{K'})\cong(\CX_K,\frac19\CD_K)\times_KK'\quad \text{and}\quad (\CX_0',\frac19\CD_0')\in\CY^{\rm{K}}(\Ik),$$  where $K'=K(T')$. Then it follows from Remark \ref{Remark. property of S-equiv}(3) that $(X,\frac19 D)$ is S-equivalent to $(\CX_0',\frac19\CD_0')$. Hence by \cite[Theorem 1.3]{LWX21}, $(X,\frac19 D)$ is S-equivalent to a pair $(X_1,\frac19 D_1)\in\CY^{\mathrm{K}}(\Ik)$ such that $(X_1,\frac{1-\vep}{9}D_1)$ is K-polystable.

The classification in \cite[Theorem 3.3 and Corollary 3.8]{Zhao24} shows that $(X_1,\frac19D_1)$ is isomorphic either to a pair in (1) or to
\[
    \big((x_1x_2x_3=x_0^3),\frac19(x_0^9=0)\big)\subset\IP^3_{[x_0:x_1:x_2:x_3]}.
\]
In the latter case, consider the weakly special test configuration
\begin{equation}
    \label{Equation. Kps point to 3planes}
        \big((x_1x_2x_3-tx_0^3=0),\frac19(x_0^9=0)\big)\subset\IP^3_{[x_0:x_1:x_2:x_3]}\times\IA^1_t\longrightarrow\IA^1_t.
\end{equation}
The central fiber is $(X_0,\frac19D_0)$. Hence $(X_0,\frac19D_0)$ is in $\CY^{\cy}(\Ik)$ and S-equivalent to $(X_1,\frac19D_1)$. 
\end{proof}
\begin{cor}
    \label{Corollary. Only Type 1 and 3}
    Every pair $(X,\frac19 D)$ in $\CY^{\cy}(\Ik)$ is of Type \uppercase\expandafter{\romannumeral1} or Type \uppercase\expandafter{\romannumeral3}.
\end{cor}
\begin{proof}
    The pairs in the first case of Proposition \ref{Proposition. K S-equiv CY} are all klt by \cite[Proposition 5.8]{GKS21}, thus of Type \uppercase\expandafter{\romannumeral1}. Since each component of the normalization of $(X_0,\frac19D_0)$ is dlt with a 0-dimensional minimal lc center, it is of Type \uppercase\expandafter{\romannumeral3}. The assertion therefore follows from Remark \ref{Remark. property of S-equiv}(2).
\end{proof}
\begin{lem}
\label{Lemma. 2-complement}
    Let $(X,D)$ be a Type \uppercase\expandafter{\romannumeral3} boundary polarized CY surface pair. There exists an effective $\IQ$-divisor $B$ on $X$ such that $(X,B)$ is slc and $2(K_X+B)\sim0$.
\end{lem}
\begin{proof}
    By \cite[Theorem 5.9]{BL24}, there exists a weakly special degeneration $(\CX,\CD)\to \IA^1$ of $(X,D)$ such that the normalization of $(\CX_0,\CD_0)$ is a union of toric pairs. On every component of the normalization of $(\CX_0,\CD_0)$, the conductor together with the strict transform of $\CD_0$ is the reduced toric boundary. Hence $\CD_0$ is an integral Weil divisor. Since $(\CX_0,\CD_0)$ is of Type \uppercase\expandafter{\romannumeral3},  \cite[Theorem B.1]{ABB23} implies that 
    \[
    2(K_{\CX_0}+\CD_0)\sim 0.
    \]
 Since $\CD_0$ is $\IG_m$-invariant, \cite[Lemma 12.2]{ABB23} gives a $\IG_m$-equivariant $\IQ$-divisor $\CB$ on $\CX$ such that $(\CX,\CB)\to\IA^1$ is a family of boundary polarized CY pairs with $2(K_{\CX/\IA^1}+\CB)\sim0$. Let $B$ be the restriction of $\CB$ to the general fiber. Then $(X,B)$ is slc and $2(K_X+B)\sim 0$.
\end{proof}
\begin{prop}
    \label{Proposition. Extension on surfaces}
    Let $S$ be a smooth surface essentially of finite type over $\Ik$, $0\in S$ a closed point, and $S^{\circ}:=S\backslash\{0\}$. If $f:(X,\frac19D)\to S$ is a family in $\CY^{\cy}$ such that its restriction $(X^{\circ},\frac19D^{\circ})\to S^{\circ}$ is in $\CY_1^{\cy}$, then $(X,\frac19D)\to S$ is in $\CY_1^{\cy}$.
\end{prop}
\begin{proof}
    Assume that $K_{X_0}$ has Cartier index $r>1$ at a closed point $x_0$. If $x_0$ is not an lc center of $(X_0,0)$, then  \cite[Theorem 3.5]{BL24} gives a curve $x_0\in C\subset X$ not contained in $X_0$ with
    \[
        \text{ind}_{x_0}(K_{X_0})=\text{ind}_x(K_{X_s})
    \]
    for all $x\in C$ and $s=f(x)$. This contradicts the assumption that $r>1$. Thus $x_0$ is an lc center of $(X_0,0)$. Since $(X_0,\frac19D_0)$ is slc, $x_0\notin\Supp(D_0)$. Thus $9K_{X_0}$ is Cartier at $x_0$ by $9K_{X_0}+D_0\sim 0$. 

    The 0-dimensional lc center $x_0$ implies that $(X_0,\frac19D_0)$ is non-klt, hence of Type \uppercase\expandafter{\romannumeral3} from Corollary \ref{Corollary. Only Type 1 and 3}. Then Lemma \ref{Lemma. 2-complement} yields an slc pair $(X_0,B_0)$ with $2(K_{X_0}+B_0)\sim 0$. Hence $x_0\notin\Supp(B_0)$ and $2K_{X_0}$ is Cartier at $x_0$. Thus $r|9$ and $r|2$, which also contradicts the assumption that $r>1$. Therefore $(X,\frac19D)\to S$ is in $\CY_1^{\cy}$.
\end{proof}
\begin{thm}
\label{Theorem. Existence of good moduli}
    The index one CY-moduli stack of unmarked cubic surfaces $\CY_1^{\cy}$ admits a good moduli space morphism
    \[
    \CY_1^{\cy}\longrightarrow Y_1^{\cy},
    \]
    where $Y_1^{\cy}$ is a separated algebraic space of finite type.
\end{thm}
\begin{proof}
    Let $\CM_1\subset\CM:=\CM(\chi,9,\frac19)$ be the substack parametrizing pairs $(X,\frac19D)$ with $K_X$ Cartier. By \cite[Proposition 3.3]{BL24}, $\CM_1$ is a finite type algebraic stack with affine diagonal. Note that for a finite type algebraic stack with affine diagonal, its arbitrary closed substack and seminormalization are both finite type algebraic stacks with affine diagonal. Hence it follows from Construction \ref{Construction. moduli Ycy} that $\CY_1^{\cy}$ is a finite type algebraic stack with affine diagonal. Thus it remains to verify S-completeness and $\Theta$-reductivity of $\CY_1^{\cy}$.

    Let $R$ be a DVR essentially of finite type over $\Ik$ with uniformizer $\pi$. Set
    \[
        S:=\Spec(R[s,t]/(st-\pi))\quad\text{or}\quad S:=\Spec(R[t]),
    \]
    where $\IG_m$ acts on $S$ with weight 1 and -1 on $s$ and $t$ respectively. Let $0\in S$ denote the unique closed fixed point and $S^{\circ}=S\backslash\{0\}$. Thus $[S/\IG_m]$ equals $\overline{\text{ST}}_R$ or $\Theta_R$. 

    Let $(X^{\circ},D^{\circ})\to S^{\circ}$ be a $\IG_m$-equivariant family in $\CY_1^{\cy}$. By Theorem \ref{Theorem. bpCY stack}, the family extends uniquely to a $\IG_m$-equivariant family $(X,D)\to S$ in $\CM$. Since $S$ is smooth, Construction \ref{Construction. moduli Ycy} implies that the morphism $S\to\CM$ factors through $\CY^{\cy}\to\CM$. Thus $(X,D)\to S$ is in $\CY^{\cy}$. Hence by Proposition \ref{Proposition. Extension on surfaces}, $(X,D)\to S$ is in $\CY_1^{\cy}$. Therefore $\CY_1^{\cy}$ is S-complete and $\Theta$-reductive with respect to DVRs essentially of finite type over $\Ik$. Then the assertion follows from Theorem \ref{Thereom. S Theta criterion}. 
    \end{proof}
For the existence of good moduli spaces of $\CY^{\mathrm{K}}$ and $\CY^{\ksba}$, we have the following results. 
\begin{prop}
\label{Proposition. good moduli of K and KSBA}
    \leavevmode
    \begin{enumerate}
        \item The stacks $\CY^{\mathrm{K}}$ and $\CY^{\ksba}$ are both finite type open substacks of $\CY_1^{\cy}$;
        \item There exists a good moduli space morphism $\phi_{\mathrm{K}}:\CY^{\mathrm{K}}\to Y^{\mathrm{K}}$ to a projective scheme;
        \item There exists a coarse moduli space morphism $\phi_{\ksba}:\CY^{\ksba}\to Y^{\ksba}$ to a projective scheme.
    \end{enumerate}
\end{prop}
\begin{proof}
    The proof is the same as \cite[Theorem 2.37 and Theorem 2.39]{BL24}.
\end{proof}
\begin{rmk}
    Since $\CY^{\ksba}$ is defined over a field of characteristic 0 and has finite stabilizers \cite[Proposition 8.64]{Kollar23}, it has linearly reductive stabilizers. Hence $\CY^{\ksba}$ is tame by \cite[Theorem 3.2]{AOV08}, and then $\phi_{\ksba}$ is a good moduli space morphism. 
\end{rmk}
\begin{rmk}
     \label{Remark. Comparing two K/KSBA}
     The compactifications $Y^{\mathrm{K}}$ and $Y^{\ksba}$ have been studied previously. On the K-moduli side, $Y^{\mathrm{K}}$ recovers the K-moduli space $\overline{M}^{\mathrm{K}}_{3,(1-\vep)/9}$ for $0<\vep\ll1$ studied in \cite{Zhao24}, where it was proved that there are no K-moduli walls. Consequently, $Y^{\mathrm{K}}$ is isomorphic to the GIT moduli space of cubic surfaces.
 
     On the KSBA-moduli side, let $N$ denote Naruki's compactification of the moduli space of marked cubic surfaces \cite{Naruki82}. The corresponding unmarked moduli space is obtained as the quotient of $N$ by the action of the Weyl group $W(E_6)$, which permutes the marking of the lines. This quotient is recovered by $Y^{\ksba}$. Naruki's compactification $N$ has recently been studied from the KSBA perspective in \cite{GKS21} and \cite{FSW25}, and explicit KSBA-moduli wall crossings for marked cubic surfaces were studied in \cite{Schock24}. 
\end{rmk}
The universal property of good moduli spaces gives unique morphisms
\[
    \pi_{\mathrm{K}}:Y^{\mathrm{K}}\longrightarrow Y_1^{\cy}\quad\text{and}\quad \pi_{\ksba}:Y^{\ksba}\longrightarrow Y_1^{\cy}
\]
making the following diagram commutative:
\begin{equation}
\label{Equation. wall crossing diagram}
\begin{tikzcd}
	{\CY^{\mathrm{K}}} & {\CY_1^{\cy}} & {\CY^{\ksba}} \\
	{Y^{\mathrm{K}}} & {Y_1^{\cy}} & {Y^{\ksba}}
	\arrow[hook, from=1-1, to=1-2]
	\arrow["{\phi_{\mathrm{K}}}", from=1-1, to=2-1]
	\arrow["{\phi_{\mathrm{\cy}}}", from=1-2, to=2-2]
	\arrow[hook', from=1-3, to=1-2]
	\arrow["{\phi_{\mathrm{\ksba}}}", from=1-3, to=2-3]
	\arrow["{\pi_{\mathrm{K}}}", from=2-1, to=2-2]
	\arrow["{\pi_{\ksba}}"', from=2-3, to=2-2]
\end{tikzcd}
\end{equation}

In Section \ref{Section. sec 5}, we will describe the morphisms $\pi_{\mathrm{K}}$ and $\pi_{\ksba}$ explicitly.
\subsection{Classification of closed points}
We now classify the closed points of $\CY_1^{\cy}(\Ik)$.
\begin{lem}
    \label{Lemma. Automorphism group}
     For the pair $(X_0,\frac19 D_0)$ in (\ref{Equation. def of p0}), we have
    \[
        \Aut(X_0,\frac19 D_0):=\{g\in\Aut(X_0)|g^*D_0=D_0\}\cong (\IG_m)^3\rtimes S_3.
    \]
\end{lem}
\begin{proof}
    Since the anticanonical linear system $|-K_{X_0}|$ embeds $X_0$ into $\IP^3$, the automorphisms of $X_0$ are induced by automorphisms of $\IP^3$. After fixing the order of the three plane components of $X_0$, a projective transformation preserving $X_0$ has the form
    \[
        [x_0:x_1:x_2:x_3]\mapsto [\lambda_0x_0+l(x_1,x_2,x_3):\lambda_1x_1:\lambda_2x_2:\lambda_3x_3],
    \]
    where $\lambda_i\in \Ik^{\times}$ and $l(x_1,x_2,x_3)$ is a linear form in $x_1,x_2$ and $x_3$. Since it preserves $D_0$, $l(x_1,x_2,x_3)=0$. Modulo a common scalar, the connected component is $(\IG_m)^3$. Moreover, $S_3$ acts on $(X_0,\frac19D_0)$ by permuting the three irreducible components of $X_0$, hence $\Aut(X_0,\frac19 D_0)\cong (\IG_m)^3\rtimes S_3.$ 
\end{proof}
The point $p_0=[(X_0,\frac19D_0)]$ lies in neither $\CY^{\rm{K}}(\Ik)$ nor $\CY^{\ksba}(\Ik)$. It is not in $\CY^{\rm{K}}(\Ik)$ because a K-semistable log Fano pair is klt, in particular  normal, by \cite{Odaka13}. It is not in $\CY^{\ksba}(\Ik)$ because the coefficient of $\frac{1+\vep}{9}D_0$ is larger than 1. 
\begin{prop}
\label{Proposition. Unique closed pt}
    The point $p_0$ is closed in $\CY_1^{\cy}(\Ik)$. Moreover, for a closed point $x\in\CY_1^{\cy}(\Ik)$, we have the following classification:
    \begin{enumerate}
        \item If $x\in \CY^{\rm{K}}(\Ik)\cup\CY^{\ksba}(\Ik)$, then $x=[(X,\frac19\sum_iL_i)]$ where $X$ is a GIT-stable cubic surface and $L_i$'s are the 27 lines on it.
        \item If $x\notin\CY^{\rm{K}}(\Ik)\cup\CY^{\ksba}(\Ik)$, then $x=p_0$.
    \end{enumerate}
\end{prop}
\begin{proof}
    We first prove that $p_0$ is a closed point in $\CY_1^{\cy}(\Ik)$. For any $[(X,\frac19 D)]\in \CY_1^{\cy}(\Ik)$, the identity component $\Aut^0(X,\frac19D)\cong(\IG_m)^r$ for some $r\geq0$ by \cite[Theorem 6.5]{ABB23}. Since $\chi(X,-mK_X)=\frac32(m^2+m)+1$, we have $(-K_X)^2=3$. Applying \cite[Proposition 3.5]{ABB23}, we obtain that $H^1(X,-mK_X)=0$ for $m\in\IN$. Thus it follows from \cite{Fujita90} and \cite[Corollary 4.10]{Reid94} that $-K_X$ is very ample and the linear system $|-K_X|$ induces an embedding $X\hookrightarrow\IP^3$. Hence $\Aut^0(X,\frac19D)$ is a subgroup of $\rm{PGL}_4$, which implies $r\leq3$. 

    Let $(\CX,\frac19\CD)\to\IA^1$ be a weakly special degeneration of $(X_0,\frac19D_0)$. If $(X_0,\frac19D_0)\ncong(\CX_0,\frac19\CD_0)$, then \cite[Proposition 6.6]{ABB23} implies that
    \[
        \dim\Aut(\CX_0,\frac19\CD_0)>\dim\Aut(X_0,\frac19D_0)=3,
    \]
 contradicting the previous bound $r\leq 3$. Hence $(X_0,\frac19D_0)\cong(\CX_0,\frac19\CD_0)$. Then it follows from \cite[Proposition 6.13]{ABB23} that $p_0$ is closed. By \cite[Theorem 6.15]{ABB23}, the points $[(X,\frac19\sum_i L_i)]$ in (1) are all closed. Hence the classification follows from  Proposition \ref{Proposition. K S-equiv CY}.
\end{proof}

\section{A local quotient presentation}
\label{Section. sec 4}
In this section, we construct a local quotient presentation of $\CY_1^{\cy}$ at $p_0$, which will serve as the local model for the wall crossing studied in Section \ref{Section. sec 5}. 
\subsection{Construction of an explicit family}
 We construct a seven-dimensional family of boundary polarized CY pairs in $\CY_1^{\cy}$. Set
\[
    U=\Spec\Ik[t,P_1,Q_1,P_2,Q_2,P_3,Q_3]\cong\IA^7,
\]
and consider a family of cubic surfaces $f:\CX\to U$ defined by 
\begin{equation}
\label{Equation. Family on U}
\CX:=(x_1x_2x_3=tP)\subset\IP^3_{[x_0:x_1:x_2:x_3]}\times U,\quad P=x_0^3+\sum_{i=1}^3(P_ix_0x_i^2+Q_ix_i^3).
\end{equation}
By Lemma \ref{Lemma. Automorphism group}, set 
 \[
     G=(\IG_m)^3\rtimes S_3\cong\Aut(X_0,\frac19 D_0).
 \]
 For $(\lambda_1,\lambda_2,\lambda_3)\in (\IG_m)^3$, we define its action on $\IP^3$ and $U$ by
 \[
     [x_0:x_1:x_2:x_3]\mapsto[x_0:\lambda_1x_1:\lambda_2x_2:\lambda_3x_3],
 \]
 \[
 t\mapsto\lambda_1\lambda_2\lambda_3t,\quad P_i\mapsto\lambda_i^{-2}P_i,\quad Q_i\mapsto\lambda_i^{-3}Q_i,\quad i=1,2,3.
 \]
 The symmetric group $S_3$ acts on $\IP^3\times U$ by permuting the triples $(x_1,P_1,Q_1),(x_2,P_2,Q_2)$ and $(x_3,P_3,Q_3)$ simultaneously. Then the family $f:\CX\to U$ is flat and $G$-equivariant.
 
 Let $U^{\rm{sm}}\subset U$ be the open locus with smooth fibers. Consider the relative Fano scheme of lines  
 \begin{equation}
     \label{Equation. Fano of lines}
     F_1(\CX^{\rm{sm}}/U^{\rm{sm}})\longrightarrow U^{\rm{sm}}.
 \end{equation}
  The morphism is finite \'{e}tale of degree 27. Indeed, let $L$ be a line on a smooth cubic surface $S$. Since $N_{L/S}\cong\CO_{\IP^1}(-1)$, we have $H^0(L,N_{L/S})=H^1(L,N_{L/S})=0$. By the deformation theory of the relative Hilbert scheme, (\ref{Equation. Fano of lines}) is \'{e}tale. It is also projective and quasi-finite, hence finite, with degree 27. We denote its scheme-theoretic closure by
  \[
      \overline{F}:=\overline{F_1(\CX^{\mathrm{sm}}/U^{\mathrm{sm}})}\subset \Gr(2,4)\times U.
  \]
For a geometric point $u\in \IV(t)\subset U$, the fiber is $\CX_u\cong X_0$. Denote the irreducible components and double curves of $X_0$ by
\[
    H_i:=\{x_i=0\},\quad \D_i:=\{x_j=x_k=0\},\quad \forall\{i,j,k\}=\{1,2,3\}.
\]
We use the homogeneous coordinates $[x_0:x_i]$ on $\D_i$. Set
\[
    q_{u,i}=x_0^3+P_i(u)x_0x_i^2+Q_i(u)x_i^3\in\Ik[x_0,x_i],\quad Z_{i}(u)=\IV(q_{u,i})\subset\D_i,
\]
where $P_i(u)$ and $Q_i(u)$ are the corresponding coordinate components of $u$. Each $Z_i(u)$ has length 3 and avoids $p=[1:0:0:0]$. Fix $\{i,j,k\}=\{1,2,3\}$. A point $(\lambda,\mu)\in Z_i(u)\times Z_j(u)$ determines the line
\[
L_{ij}(\lambda,\mu)=\IV(x_k,x_0-\lambda x_i-\mu x_j)\subset H_k
\]
joining $[\lambda:1]\in\D_i$ and $[\mu:1]\in\D_j$. This yields a degree 9 divisor
\[
    J_{ij}(u):=\sum_{\lambda,\mu}\mult_{\lambda}(Z_i(u))\mult_{\mu}(Z_j(u))L_{ij}(\lambda,\mu)\subset H_k,
\]
where $(\lambda,\mu)$ ranges over all points in $Z_i(u)\times Z_j(u)$.
\begin{lem}
\label{Lemma. barF deg 27}
    After a $G$-invariant shrinking of $U$ around $\IV(t)$, the morphism $\overline{F}\to U$ is finite flat of degree 27. For every geometric point $u\in\IV(t)$, there is a scheme-theoretic decomposition
    \begin{equation}
    \label{Equation. decompose Fu}
        \oF_u\cong(Z_1(u)\times Z_2(u))\sqcup(Z_1(u)\times Z_3(u))\sqcup(Z_2(u)\times Z_3(u)).
    \end{equation}
\end{lem}
\begin{proof}
    Fix $\{i,j,k\}=\{1,2,3\}$ and consider the affine chart $\mathscr{U}_{ij}\subset\Gr(2,4)$ whose lines are written uniquely as
\begin{equation}
\label{Equation. equation of lines}
    L_{\alpha,\beta,\lambda,\mu}:\quad x_k=\alpha x_i+\beta x_j,\quad x_0=\lambda x_i+\mu x_j.
\end{equation}
It lies in $H_k$ precisely when $\alpha=\beta=0$. Substitution into $P$ gives
\[
    P|_{L_{\alpha,\beta,\lambda,\mu}}=\Psi_0x_i^3+\Psi_1x_i^2x_j+\Psi_2x_ix_j^2+\Psi_3x_j^3,
\]
where 
\[
    \Psi_0=\lambda^3+P_1\lambda+Q_1+P_3\lambda\alpha^2+Q_3\alpha^3,
\]
\[
    \Psi_1=3\lambda^2\mu+P_1\mu+P_3(2\lambda\alpha\beta+\mu\alpha^2)+3Q_3\alpha^2\beta,
\]
\[
    \Psi_2=3\lambda\mu^2+P_2\lambda+P_3(\lambda\beta^2+2\mu\alpha\beta)+3Q_3\alpha\beta^2,
\]
\[
    \Psi_3=\mu^3+P_2\mu+Q_2+P_3\mu\beta^2+Q_3\beta^3.
\]
Since $x_1x_2x_3|_{L_{\alpha,\beta,\lambda,\mu}}=\alpha x_i^2x_j+\beta x_ix_j^2$, the restriction of $x_1x_2x_3=tP$ is
\[
    -t\Psi_0x_i^3+(\alpha-t\Psi_1)x_i^2x_j+(\beta-t\Psi_2)x_ix_j^2-t\Psi_3x_j^3=0.
\]
Over $\{t\neq0\}\subset U$, the line lies on the cubic surfaces exactly when
\[
    \Psi_0=0,\quad\alpha=t\Psi_1,\quad \beta=t\Psi_2,\quad\Psi_3=0.
\]
Set
\begin{equation}
\label{Equation. Ideal of lines}
    Z_{ij}=\IV(\alpha-t\Psi_1,\beta-t\Psi_2,\Psi_0,\Psi_3)\subset\mathscr{U}_{ij}\times U.
\end{equation}
Reducing (\ref{Equation. Ideal of lines}) at $u$ gives $\alpha=\beta=0$ and $q_{u,i}([\lambda:1])=q_{u,j}([\mu:1])=0$. Thus
\[
    (Z_{ij})_u\cong\Spec\frac{\kappa(u)[\lambda,\mu]}{(q_{u,i}([\lambda:1]),q_{u,j}([\mu:1]))}=Z_i(u)\times Z_j(u).
\]
 The four generators in (\ref{Equation. Ideal of lines}) form a regular sequence along $\IV(t)$. Hence $Z_{ij}\to U$ is Cohen-Macaulay of relative dimension 0 at every point over $\IV(t)$. Thus it is flat at every such point. By construction, 
 \[
     F_1(\CX^{\mathrm{sm}}/U^{\mathrm{sm}})\cap(\mathscr{U}_{ij}\times U^{\mathrm{sm}})=Z_{ij}\times_UU^{\mathrm{sm}}.
 \]
 By taking closure, $\oF\cap(\mathscr{U}_{ij}\times U)=Z_{ij}$ near fibers over $\IV(t)$.
 \begin{claim}
     For $u\in\IV(t)$, no line represented by $\oF_u$ passes through $p=[1:0:0:0]$.
 \end{claim}
 \begin{proof}[Proof of claim]
     Suppose such a line $L_0$ exists and choose a DVR family of lines $\CL_R\subset \CX_R$ over $\Spec R$, whose special fiber is $L_0$ and whose generic fiber lies on a smooth cubic surface. We may assume $x_1|_{L_0}\neq0$. Since $p\in L_0$ and $x_1(p)=0$, we have $L_0\cap\{x_1=0\}=p$. Thus the intersection $\CL_R\cap\{x_1=0\}$ is a section $q:\Spec R\to\CL_R$ specializing to $p$. Let $\eta$ be the generic point of $\Spec R$, then $tP(q(\eta))=0$. Since $\CX_{\eta}$ is smooth, we have $P(q(\eta))=0$, implying $P_u(p)=0$. However, a direct computation shows that $P_u(1,0,0,0)=1$, which gives a contradiction.
 \end{proof}
 With this claim, we have $\oF_u=(Z_{12})_u\sqcup (Z_{13})_u\sqcup (Z_{23})_u$ for $u\in\IV(t)$, implying the decomposition (\ref{Equation. decompose Fu}). Thus $\oF\to U$ is quasi-finite and flat at every point over $\IV(t)$. Since $\oF\to U$ is proper and $G$-equivariant, the image of non-quasi-finite locus and the image of non-flat locus are both $G$-invariant closed subsets disjoint from $\IV(t)$. Therefore, after a $G$-invariant shrinking of $U$ around $\IV(t)$, $\oF\to U$ is finite flat of degree 27. 
\end{proof}

Let $\CL_{\Gr}\subset\IP^3\times \Gr(2,4)$ be the universal family over $\Gr(2,4)$. Set $\CL:=\CL_{\Gr}\times_{\Gr(2,4)}\oF\subset \IP^3\times \oF$. Since the relative Fano scheme of lines $F_1(\CX/U)\subset\Gr(2,4)\times U$ is closed and contains $F_1(\CX^{\rm{sm}}/U^{\rm{sm}})$, it also contains $\oF$. Hence $\CL\subset \CX\times_U \oF$ scheme-theoretically. By Lemma \ref{Lemma. barF deg 27}, the projection $q:\CX\times_{U}\oF\to \CX$ is a finite flat morphism after shrinking $U$ around $\IV(t)$. Thus
$$\CD:=q_*[\CL]$$
is an effective Weil divisor on $\CX$. 
\begin{lem}
\label{Lemma. D is cartier}
    After replacing $U$ by a $G$-invariant shrinking satisfying Lemma \ref{Lemma. barF deg 27}, $\CD$ is a relative effective Cartier divisor on $\CX/U$.
\end{lem}
\begin{proof}
    Consider the affine chart $\CX_1:=\{x_1\neq0\}\cap\CX$ for the family (\ref{Equation. Family on U}) without any shrinking of $U$. Put $s=P(x_0,1,x_2,x_3)$. Since $Q_1$ has coefficient 1 in $s$, we have 
    \[
        \CX_1\cong\Spec A,\quad A=\Ik[x_0,x_2,x_3,t,s,P_1,P_2,Q_2,P_3,Q_3]/(x_2x_3-ts).
    \]
     Since $\Spec A$ is a hypersurface, it is Cohen-Macaulay, thus satisfies Serre's $S_2$ condition. Hence $\Spec A$ is normal because $\mathrm{Sing}(\Spec A)=\IV(x_2,x_3,t,s)$ is of codimension 3. Let $E_2=\IV(t,x_2)$ and $E_3=\IV(t,x_3)$ be two prime divisors. Then $\mathrm{Cl}(\CX_1)$ is generated by $E_2$ and $E_3$ with relation $\div(t)=E_2+E_3$. Take a $G$-invariant shrinking $\widetilde{U}\subset U$ satisfying Lemma \ref{Lemma. barF deg 27} and let $\overline{\CD}_1$ be the closure of $\CD|_{\CX_1\times_U\widetilde{U}}$ in $\CX_1$ as a Weil divisor. The involution
    \[
        \CX_1\longrightarrow\CX_1,\quad (x_2,x_3,P_2,Q_2,P_3,Q_3)\mapsto(x_3,x_2,P_3,Q_3,P_2,Q_2)
    \]
    preserves the family and $\overline{\CD}_1$, while it exchanges $E_2$ and $E_3$. Thus $[\overline{\CD}_1]=-[\overline{\CD}_1]$ in $\mathrm{Cl}(\CX_1)$. Since $\mathrm{Cl}(\CX_1)\cong\IZ$ is torsion-free, $[\overline{\CD}_1]=0$. Then $\overline{\CD}_1$ is principal, thus its restriction $\CD|_{\CX_1\times_U\widetilde{U}}$ is Cartier. The same argument applies on $x_2\neq 0$ chart and $x_3\neq 0$ chart.
    
    The remaining points of $\CX_{\widetilde{U}}$ satisfy $x_1=x_2=x_3=0$, so $x_0\neq 0$. At such a point, 
 \[
     \frac{\partial(x_1x_2x_3-tP)}{\partial t}=-P=-x_0^3\neq 0,
 \]
  and therefore $\CX$ is smooth there. Hence $\CD$ is also Cartier there. Therefore $\CD$ is Cartier on $\CX_{\widetilde{U}}$. Since  $\Supp(\CD)\cap\CX_u\subset q(\CL_u)$ is a finite union of lines, $\Supp(\CD)$ contains no irreducible components of a fiber. Hence $\CD$ is a relative effective Cartier divisor on $\CX_{\widetilde{U}}/\widetilde{U}$.
\end{proof}
\begin{prop}
\label{Proposition. The family of bpCY on U0}
There exists a $G$-invariant open neighborhood $\IV(t)\subset U^{\circ}\subset U$ such that
\[
    f:(\CX_{U^{\circ}},\frac19\CD_{U^{\circ}})\longrightarrow U^{\circ}
\]
is a family of boundary polarized CY pairs with the following properties:
\begin{enumerate}
    \item The underlying surface of each fiber is Gorenstein.
    \item For $u\in U^{\circ}$ with $\CX_u$ smooth, $\CD_u$ is the sum of the 27 lines on $\CX_u$.
    \item For $u\in \IV(t)$, $\CD_u=J_{12}(u)+J_{23}(u)+J_{13}(u)$. In particular, the fiber over $0\in U^{\circ}$ is $(X_0,\frac19 D_0)$.
\end{enumerate}
\end{prop}
\begin{proof}
    The statements (1)(2) follow from the construction of $\CX$ and $\CD$. Now we prove (3). Since $\CL\to\oF$ is a $\IP^1$-bundle and $\oF\to U$ is flat, $\CL\to U$ is flat. By Lemma \ref{Lemma. D is cartier}, we may assume that $\CD$ is relative Cartier after replacing $U$ by a $G$-invariant shrinking, thus $\CD\to U$ is also flat. Therefore restricting $q_*[\CL]=\CD$ to a geometric fiber gives $$\CD_u=(q_u)_*[\CL_u].$$
   Fix $u\in\IV(t)$. Since $\CL\subset\CX\times_U\oF$ is the universal family of lines, the reduced irreducible component of $\CL_u$ lying over $(\lambda,\mu)\in\oF_u$ is naturally identified with $L_{ij}(\lambda,\mu)$, and the evaluation map $q_u:\CL_u\to\CX_u$ restricts to identity under this identification. Hence by Lemma \ref{Lemma. barF deg 27}, 
   \[
   \CD_u=(q_u)_*[\CL_u]=J_{12}(u)+J_{23}(u)+J_{13}(u).
   \]
    Set $\CM=\CO_{\CX}(\CD)\otimes\CO_{\CX}(-9)$. Then $\CM|_{\CX_u}\cong\CO_{\CX_u}$ for any $u\in\IV(t)$. The sequence
 \[
     0\longrightarrow\CO_{\IP^3}(-3)\longrightarrow\CO_{\IP^3}\longrightarrow\CO_{\CX_u}\longrightarrow 0
 \]
 gives $H^1(\CX_u,\CO_{\CX_u})=H^2(\CX_u,\CO_{\CX_u})=0$ and $H^0(\CX_u,\CO_{\CX_u})\cong\Ik$. Thus by a further $G$-invariant shrinking of $U$ around $\IV(t)$, we may assume that $R^if_*\CM=0$ for $i>0$ and $M:=f_*\CM$ is a line bundle on $U$. The evaluation map $f^*M\to\CM$ is an isomorphism on $\CX_u$ for any $u\in\IV(t)$. Hence after a $G$-invariant shrinking of $U$ around $\IV(t)$, we may assume that $f^*M\cong \CM$. By adjunction, $\omega_{\CX/U}\cong\CO_{\CX}(-1)$.  Consequently,
  \[
      f^*M\cong\CO_{\CX}(\CD)\otimes\CO_{\CX}(-9)\cong\omega_{\CX/U}^{\otimes9}(\CD).
  \]
  Hence $K_{\CX/U}+\frac19\CD\sim_{\IQ,U}0$ and $\CO_{\CX}(\CD)\cong_{U}\CO_{\CX}(9)$ is relatively ample. Let $\Sigma=\rm{Sing}(\CX/U)$ be the relative singular locus. Then $\dim(\CD_u\cap\Sigma_u)\leq0$ for $u\in \IV(t)$. Then by upper semicontinuity, after a $G$-invariant shrinking of $U$ around $\IV(t)$, we may assume that it holds for any $u\in U$. Then $\CX_u$ is smooth at the generic points of $\CD_u$ and $\CD_u$ contains no irreducible components of $\CX_u$. Hence $\CD$ is a relative Mumford divisor on $\CX/U$. Since $(\CX_u,\frac19\CD_u)$ is slc for any $u\in \IV(t)$, we may assume that it holds for $u\in U$ by the openness of slc singularities.
 
  The above steps give a $G$-invariant open neighborhood $\IV(t)\subset U^{\circ}\subset U$. Thus $(\CX_{U^{\circ}},\frac19\CD_{U^{\circ}})\to U^{\circ}$ is a family of boundary polarized CY pairs satisfying (1)(2)(3).
\end{proof}
 
 \begin{lem}
     \label{Lemma. Locally quotient}
     There exists an open substack $\CY_1^{\circ}\subset\CY_1^{\cy}$ such that:
     \begin{enumerate}
         \item $(\CX_{U^{\circ}},\frac19\CD_{U^{\circ}})\to U^{\circ}$ is a family in $\CY_1^{\circ}$.
         \item $\CY_1^{\circ}\cong[A/\mathrm{PGL}_4]$ for some irreducible seminormal scheme $A$ of finite type  over $\Ik$ with a $\mathrm{PGL}_4$-action.
     \end{enumerate}
 \end{lem}
 \begin{proof}
     Let $\CM^{\circ}\subset\CM=\CM(\chi,9,\frac19)$ denote the full subcategory consisting of families $(X,\frac19D)\to T$ such that $-K_{X_t}$ is very ample and $\deg(D_t)=27$ for any $t\in T$. It is an open substack by \cite[Proposition 3.5]{ABB23}. From \cite[Proposition 3.9]{ABB23}, there exists a finite type scheme $H$ with a $\mathrm{PGL}_4$-action and an isomorphism $\CM^{\circ}\cong[H/\mathrm{PGL}_4]$. Consider the Cartesian diagram
    \[\begin{tikzcd}
	{A=H\times_{\CM^{\circ}}\CY_1^{\circ}} & {\mathcal{Y}_1^{\circ}=\CM^{\circ}\times_{\CM}\CY_1^{\cy}} & {\mathcal{Y}_1^{\rm{CY}}} \\
	H & {\mathcal{M}^{\circ}=[H/\mathrm{PGL}_4]} & {\mathcal{M}}
	\arrow[from=1-1, to=1-2]
	\arrow[from=1-1, to=2-1]
	\arrow[from=1-2, to=1-3]
	\arrow[from=1-2, to=2-2]
	\arrow[from=1-3, to=2-3]
	\arrow[from=2-1, to=2-2]
	\arrow[hook, from=2-2, to=2-3]
\end{tikzcd}\]
Here the morphism $\CY_1^{\cy}\to\CM$ is given by (\ref{Equation. construct Y1cy}). Then $\CY_1^{\circ}\cong[A/\mathrm{PGL}_4]$ is an open substack of $\CY_1^{\cy}$ and $A$ is a seminormal finite type scheme by \cite[Definition 13.6]{ABB23}. Since $-K_{\CX_u}\sim_{\IQ}\CO_{\CX_u}(1)$ is very ample for every $u\in U^{\circ}$, $(\CX_{U^{\circ}},\frac19\CD_{U^{\circ}})$ is a  family in $\CM^{\circ}$. Hence assertion (1) holds.

Set $A^{\rm{sm}}=A\times_{\CY_1^{\cy}}\CY$. Here the dense substack $\CY\subset\CY_1^{\cy}$, defined in Construction \ref{Construction. moduli Ycy}, parametrizes $(X,\frac19\sum_iL_i)$ where $X$ is a smooth cubic surface and $L_i$'s are the uniquely determined 27 lines on it. Then $A^{\rm{sm}}$ is canonically isomorphic to the subscheme of $\IP(H^0(\IP^3,\CO_{\IP^3}(3)))$ parametrizing smooth cubic surfaces, which is irreducible. Since $A\to\CY_1^{\cy}$ is an open map, $A^{\rm{sm}}$ is dense in $A$. Therefore $A$ is irreducible. 
 \end{proof}
By Lemma \ref{Lemma. Locally quotient}, the $G$-equivariant family $(\CX_{U^{\circ}},\frac19\CD_{U^{\circ}})\to U^{\circ}$ induces a $\mathrm{PGL}_4$-equivariant morphism
\begin{equation}
    \label{Equation. equivariant morphism Phi}
    \Phi:\mathrm{PGL}_4\times^{G}U^{\circ}\to A.
\end{equation}
Denote $z_0=(1,0)\in \mathrm{PGL}_4\times^{G}U^{\circ}$ and $y_0=\Phi(z_0)$. The key step in constructing a local quotient presentation of $\CY_1^{\cy}$ at $p_0$ is the following proposition:
\begin{prop}
    \label{Proposition. Etaleness of Phi}
    The morphism $\Phi$ is \'{e}tale at $z_0$.
\end{prop}
\subsection{\'{E}taleness of $\Phi$}
We will use the following criterion.
\begin{lem}
\label{Lemma. Etale criterion}
    Let $f:X\to Y$ be a morphism of schemes. Let $x\in X$ with image $y\in Y$. Assume
    \begin{enumerate}
        \item $Y$ is integral and geometrically unibranch at $y$;
        \item $f$ is locally of finite type;
        \item there is a specialization $x'\rightsquigarrow x$ such that $f(x')$ is the generic point of $Y$;
        \item $f$ is unramified at $x$.
    \end{enumerate}
    Then $f$ is \'{e}tale at $x$.
\end{lem}
\begin{proof}
    See \cite[\href{https://stacks.math.columbia.edu/tag/0GS8}{Tag 0GS8}]{stacks-project}.
\end{proof}
We now verify the hypotheses of Lemma \ref{Lemma. Etale criterion} for $\Phi$ at $z_0$, thereby proving Proposition \ref{Proposition. Etaleness of Phi}.
\begin{lem}
\label{Lemma. Centered normal form}
Let $R$ be a complete DVR with fraction field $K$ and residue field $\kappa$. Let $(S,\frac19 C)\to\Spec R$ be a family in $\CY_1^{\cy}$ whose special fiber is the base change of $(X_0,\frac19D_0)$ to $\kappa$ and whose generic fiber is a smooth cubic surface with the sum of lines on it. After a projective change of coordinates over $R$, the family is pulled back from $(\CX_{U^{\circ}},\frac19\CD_{U^{\circ}})\to U^{\circ}$ by a morphism $\Spec R\to U^{\circ}$ mapping the closed point to the origin.
\end{lem}
\begin{proof}
    Let $\fm$ be the maximal ideal and $\pi$ be a uniformizer of $R$. After multiplying by a unit, we may write the defining equation of $S$ as $x_1x_2x_3=\pi H$, where $H\in R[x_0,x_1,x_2,x_3]_3$ is a degree 3 homogeneous polynomial. For linear forms $l_1,l_2,l_3\in R[x_0,x_1,x_2,x_3]_1$, one has
    \[
        (x_1+\pi^n l_1)(x_2+\pi^n l_2)(x_3+\pi^n l_3)\equiv x_1x_2x_3+\pi^n(l_1x_2x_3+l_2x_1x_3+l_3x_1x_2)\bmod{\fm^{n+1}}.
    \]
    Thus inductively we can construct projective linear transforms $\{f_n\}_{n=1}^{\infty}$ such that, after applying the composition $g_n:=f_nf_{n-1}\cdots f_1$, the defining equation of $S$ is of the form
 \begin{equation}
 \label{Equation. No mixed terms}
     x_1x_2x_3=d^{(n)}x_0^3+\sum_{i=1}^3(a^{(n)}_ix_i^3+b^{(n)}_ix_i^2x_0+c^{(n)}_ix_ix_0^2)+\pi^{n}H_n, 
 \end{equation}
 where $H_n\in R[x_0,x_1,x_2,x_3]_3$ and $d^{(n)},a^{(n)}_i,b^{(n)}_i,c^{(n)}_i\in\fm$. We may choose $f_1$ to be the identity map. Thus $g_n\in M_4(R/\fm^{n})$ and 
 \[
   g_{n+1}\equiv g_n\bmod{\fm^n},\quad a_i^{(n+1)}\equiv a_i^{(n)},b_i^{(n+1)}\equiv b_i^{(n)},c_i^{(n+1)}\equiv c_i^{(n)},d^{(n+1)}\equiv d^{(n)}\bmod{\fm^n}. 
 \]
 Since $R$ is complete, we obtain $$g:=\varprojlim g_n\in M_4(R),\quad a_i:=\varprojlim a_i^{(n)},b_i:=\varprojlim b_i^{(n)},c_i:=\varprojlim c_i^{(n)},d:=\varprojlim d^{(n)}\in\fm.$$
 Since $g\equiv \rm{id}\bmod{\fm}$, we have $\det(g)\notin\fm$. Then $g\in M_4(R)$ is invertible, thus defines an element of $\mathrm{PGL}_4(R)$. Therefore, after applying $g$, the defining equation of $S$ is of the form
 \begin{equation}
 \label{Equation. Equation after first transf}
      x_1x_2x_3=dx_0^3+\sum_{i=1}^3(a_ix_i^3+b_ix_i^2x_0+c_ix_ix_0^2),
 \end{equation}
 where $a_i,b_i,c_i,d\in\fm$. Let $v$ be the discrete valuation determined by the DVR $R$.
 \begin{claim}
     $v(d)<\min_{1\leq i\leq 3}\{v(a_i),v(b_i),v(c_i)\}$.
 \end{claim} 

\begin{proof}[Proof of Claim]
     Let $F_1(S_K)\subset\Gr(2,4)_K$ be the Fano scheme of lines on the smooth cubic surface $S_K$ and $F\subset\Gr(2,4)_R$ be its scheme-theoretic closure. Then $F\to \Spec R$ is finite flat of degree 27. Write $L_i=\IV(x_0,x_i)\subset X_0$ for $1\leq i\leq 3$, then $C_{\kappa}=9\sum_iL_i$. Thus $F_{\kappa}$ is supported at the three points $[L_i](1\leq i\leq 3)$, with length 9 at each. Set
     \[
         q_i(T)=a_i+b_iT+c_iT^2+dT^3,\quad e_i=\min\{v(a_i),v(b_i),v(c_i),v(d)\},\quad 1\leq i\leq 3.
     \]
     Then all of the $q_i$'s are nonzero. Indeed, if $q_1(T)=0$, then a direct computation shows that $S_K$ is singular at $[0:1:0:0]$. The same result holds for $q_2$ and $q_3$. Since $R$ is complete, $F\cong\bigsqcup_i\Spec B_i$ where $\Spec B_i$ is the part specializing to $[L_i]$. Thus $B_i$ is a finite flat local $R$-algebra of rank 9. Now we assume that $1\leq e_1\leq e_2\leq e_3$. Note that $\Spec B_3$ lies in the chart $\mathscr{U}_{12}\subset \Gr(2,4)_R$ whose lines are written as
 \[
     x_3=\alpha x_1+\beta x_2,\quad x_0=\lambda x_1+\mu x_2,
 \]
 and $\overline{B_3}:=B_3/\pi B_3$ is supported at $\alpha=\beta=\lam=\mu=0$. Substituting the equations into (\ref{Equation. Equation after first transf}), the coefficients of $x_1^2x_2$ and $x_1x_2^2$ give $\alpha,\beta\in\pi^{e_1}B_3$. The coefficients of $x_1^3$ and $x_2^3$ give
 \begin{equation}
 \label{Equation. two eqs in claim}
 \begin{split}
     q_1(\lambda)+c_3\alpha\lam^2+b_3\alpha^2\lam+a_3\alpha^3=0,\\
     q_2(\mu)+c_3\beta\mu^2+b_3\beta^2\mu+a_3\beta^3=0.
\end{split}
\end{equation}
Dividing the two equations in (\ref{Equation. two eqs in claim}) by $\pi^{e_1}$ and $\pi^{e_2}$ respectively, and passing to $\overline{B_3}$, we obtain 
\[
     A_1(\lam)=A_2(\mu)=0\text{ in }\overline{B_3},\quad A_i(T):=\overline{\pi^{-e_i}q_i(T)}\in\kappa[T]\backslash\{0\}.
\]
Let $r_i$ be the multiplicity of 0 as a root of $A_i$, for $i=1,2$.  Since $\lam$ and $\mu$ are nilpotent in $\overline{B_3}$, these relations imply that $1\leq r_i\leq 3$ and $\lam^{r_1}=\mu^{r_2}=0$. Moreover, since  $\alpha=\beta=0$ in $\overline{B_3}$, we have a surjection $\kappa[\lam,\mu]/(\lam^{r_1},\mu^{r_2})\twoheadrightarrow\overline{B_3}$. Thus $9=\dim_{\kappa}\overline{B_3}\leq r_1r_2\leq 9$, so $r_1=r_2=3$. Hence
\[
    e_1=e_2=v(d)<\min_{i=1,2}\{v(a_i),v(b_i),v(c_i)\}.
\]
Since $e_2\leq e_3\leq v(d)$, we also have $e_3=v(d)$. Repeating the same argument with $B_3$ replaced by $B_2$, we can prove the strict inequality for $a_3,b_3,c_3$. Therefore the claim is proved.
\end{proof}
Set $t=d$ and $\gamma_i=\frac{c_i}{d},P_i=\frac{b_i}{d},Q_i=\frac{a_i}{d}$, then (\ref{Equation. Equation after first transf}) can be written as
 \begin{equation}
 \label{Equation. After absorb d}
     x_1x_2x_3=t\big(x_0^3+\sum_{i=1}^3(\gamma_ix_ix_0^2+P_ix_i^2x_0+Q_ix_i^3)\big). 
 \end{equation}
Now we remove the $x_ix_0^2(1\leq i\leq 3)$ terms in (\ref{Equation. After absorb d}) inductively as before. At the $n$-th step, make the coordinate changes
\[
    x_0\mapsto x_0+\pi^n(s_1x_1+s_2x_2+s_3x_3),\quad (x_1,x_2,x_3)\mapsto (x_1+t\pi^nl_1,x_2+t\pi^nl_2,x_3+t\pi^nl_3),
\]
where $l_i\in R[x_0,x_1,x_2,x_3]_1$ are linear forms and $s_i\in R$. Modulo $t\fm^{n+1}$, the defining polynomial changes by 
\[
    t\pi^n\big(l_1x_2x_3+l_2x_1x_3+l_3x_1x_2-3(s_1x_1+s_2x_2+s_3x_3)x_0^2\big).
\]
These terms allow us to remove $x_ix_0^2(1\leq i\leq 3)$ terms and all terms divisible by $x_ix_j(1\leq i<j\leq 3)$ except $x_1x_2x_3$ terms at that order. Absorbing changes into remaining coefficients, we continue inductively. With the same argument as before, completeness of $R$ yields $h\in \mathrm{PGL}_4(R)$, changing (\ref{Equation. After absorb d}) into the form
 \[
    S':=(h\circ g)(S):\quad x_1x_2x_3=t\big(x_0^3+\sum_{i=1}^3(P_i'x_i^2x_0+Q_i'x_i^3)\big).
 \]
Thus $S'\to \Spec R$ induces a morphism $\Spec R\to U^{\circ}$. By pulling back $(\CX_{U^{\circ}},\frac19\CD_{U^{\circ}})\to U^{\circ}$ along it, we obtain a family $(\CX_{R},\frac19\CD_R)\to \Spec R$ with the generic fiber $\CX_K= S'_K$. Since $(h\circ g)(C)$ and $\CD_R$ are flat closed subschemes of $S'=\CX_R$ with the same generic fiber, $(h\circ g)(C)=\CD_R$. This proves the assertion. 
\end{proof}
Write $H=\Aut(X_0)$. Let $G^{\circ}$ and $H^{\circ}$ be the identity components of $G$ and $H$ respectively. Then
 \[
    H^{\circ}/G^{\circ}=\big\{n(s_1,s_2,s_3):x_0\mapsto x_0+\sum_{i=1}^3s_ix_i|s_i\in\Ik\big\}.
\]
 For a point $u\in \IV(t)$, the boundary divisor $\CD_u$ of $\CX_u\cong X_0$ is determined by three closed points (with multiplicities) on each double curve $\D_i$, which are the zeros of the cubic polynomials
\[
   T_i:=r_i^3+P_i(u)r_i+Q_i(u),\quad i=1,2,3.
\]
Here $r_i=x_0/x_i$. This yields a well-defined morphism
\[
  \Xi:  H^{\circ}\times^{G^{\circ}}\IV(t)\longrightarrow\prod_i\Hilb^3(\D_i),\quad [n(s_1,s_2,s_3),u]\mapsto (Z(n(s_1,s_2,s_3)\cdot T_i))_i.
\]
Here $Z(T)$ denotes the zero subscheme of a polynomial $T$. 
\begin{lem}
    \label{Lemma. Xi is open immersion}
    The morphism $\Xi$ is an open embedding.
\end{lem}
\begin{proof}
    Set $\CH_i:=\{Z\in\Hilb^3(\D_i)|[1:0]\notin\Supp(Z)\}$. It is an open subscheme of $\Hilb^3(\D_i)$ and $\Image(\Xi)\subset\prod_i\CH_i$. There is an isomorphism
    \begin{equation}
    \label{Equation. Hi iso to A3}
        \IA^3_{\Ik}\longrightarrow\CH_i,\quad (\alpha_i,\beta_i,\gamma_i)\mapsto Z(r_i^3+\alpha_ir_i^2+\beta_ir_i+\gamma_i).
    \end{equation}
    The equation of $n(s_1,s_2,s_3)\cdot T_i$ is 
    \[
        (r_i-s_i)^3+P_i(r_i-s_i)+Q_i=r_i^3-3s_ir_i^2+(P_i+3s_i^2)r_i+(Q_i-P_is_i-s_i^3).
    \]
    Thus under the isomorphism (\ref{Equation. Hi iso to A3}), $\Xi$ is defined by 
    \[
        \Xi(n(s_1,s_2,s_3),u)=(-3s_i,P_i+3s_i^2,Q_i-P_is_i-s_i^3).
    \]
    Since $\Ik$ is algebraically closed of characteristic 0, $\Xi$ is an isomorphism onto $\prod_i\CH_i$. In particular, $\Xi$ is an open embedding.
\end{proof}
Let $V=H^0(\IP^3,\CO_{\IP^3}(3))$ and $\rho:A\to \IP(V)$ be the natural morphism forgetting the boundary divisors. Let $\fh,\fg$ be the Lie algebras of $H,G$ respectively.  
\begin{lem}
\label{Lemma. Kernel of rho o Phi}
With the above notation, we have an isomorphism
    $$\ker d(\rho\circ\Phi)_{z_0}\cong(\fh\oplus T_0\IV(t))/\fg.$$
\end{lem}
\begin{proof}
     Consider the natural morphism $$T:\mathrm{PGL}_4\times U^{\circ}\to \IP(V),(g,u)\mapsto g\cdot[F_u],$$
    where $F=x_1x_2x_3-tP$ and $F_u$ is the fiber over $u\in U^{\circ}$. An element $\xi\in\mathfrak{pgl}_4$ is represented by
    \[
        x_i\mapsto x_i+\varepsilon l_i,\quad \varepsilon^2=0,\quad0\leq i\leq 3,
    \]
    where $l_i(0\leq i\leq3)$ are linear forms. A direct computation gives
    \[
        dT_{(1,0)}:\mathfrak{pgl}_4\oplus \Ik\partial_t\oplus T_0\IV(t)\longrightarrow T_{[F_0]}\IP(V)\cong V/\Ik F_0,\quad (\xi,a,v)\mapsto [x_2x_3l_1+x_1x_3l_2+x_1x_2l_3-ax_0^3].
    \]
    Here $F_0=x_1x_2x_3$. Then $(\xi,a,v)\in \ker(dT_{(1,0)})$ if and only if $\xi\in\fh$ and $a=0$. Therefore 
    $$\ker(dT_{(1,0)})= \fh\oplus T_0\IV(t).$$ 
    Quotienting by $\fg$ gives the claimed isomorphism.
\end{proof}

\begin{lem}
\label{Lemma. Unramified}
The morphism 
\[
\Phi: \mathrm{PGL}_4\times^{G}U^{\circ}\longrightarrow A
\] 
is unramified at $z_0$.
\end{lem}
\begin{proof}
Write $Y:=\mathrm{PGL}_4\times^{G}U^{\circ}$ and let $\Spec\Ik\hookrightarrow\IP(V)$ be the inclusion mapping to $[x_1x_2x_3]$. By pulling back $\rho\circ\Phi$ along it, we obtain the following Cartesian diagram
\[\begin{tikzcd}
	{Y_0} & {A_0} & {\Spec\Ik} \\
	{Y=\mathrm{PGL}_4\times^{G}U^{\circ}} & A & {\IP(V)}
	\arrow["{\Phi_0}", from=1-1, to=1-2]
	\arrow[hook, from=1-1, to=2-1]
	\arrow[from=1-2, to=1-3]
	\arrow[hook, from=1-2, to=2-2]
	\arrow[hook, from=1-3, to=2-3]
	\arrow["\Phi", from=2-1, to=2-2]
	\arrow["\rho", from=2-2, to=2-3]
\end{tikzcd}\]
Since 
 \[
     \ker(d\Phi)_{z_0}\subset\ker d(\rho\circ\Phi)_{z_0}\cong T_{z_0}Y_0,
 \]
 it suffices to prove that $\Phi_0$ is unramified at $z_0$. Since $H^\circ\times^{G^{\circ}}\IV(t)\cong H\times^{G}\IV(t)$ is a closed subscheme of $Y$ and its image under $\rho\circ\Phi$ is $[x_1x_2x_3]\in\IP(V)$, it follows that $H^\circ\times^{G^{\circ}}\IV(t)$ is a closed subscheme of $Y_0$. Denote the closed immersion by  $\iota:H^\circ\times^{G^{\circ}}\IV(t)\hookrightarrow Y_0$. The differential map at $(1,0)$
\[
    (d\iota)_{(1,0)}:(\fh\oplus T_0\IV(t))/\fg\longrightarrow T_{z_0}Y_0
\]
is injective. By $T_{z_0}Y_0\cong\ker d(\rho\circ\Phi)_{z_0}$ and Lemma \ref{Lemma. Kernel of rho o Phi}, it is an isomorphism. Consequently, it suffices to prove that $\Phi_0\circ\iota$ is unramified at $(1,0)$. 

For $1\leq i\leq 3$, let $\iota_i:\D_i\hookrightarrow X_0$ be the inclusion of the double curve of $X_0$  and $p_i$ be the intersection point $\D_i\cap(x_0=0)$. Restricting the tautological family on $A$ to $A_0$, we obtain a family $(\CS,\frac19\CC)\to A_0$ with $\CS\cong X_0\times A_0$. Consider the Cartier pullback $\CC_i:=(\iota_i\times\mathrm{id}_{A_0})^*\CC\subset \D_i\times A_0$. Then $(\CC_i)_{y_0}=D_0|_{\D_i}=9p_i$. Hence after shrinking $A_0$ around $y_0$, $\CC_i\to A_0$ is finite flat of degree 9. This induces a morphism
\[
    \mathrm{Res}:A_0\longrightarrow\prod_i\Hilb^9(\D_i),\quad a\mapsto ((\CC_i)_a)_i.
\]
Then we have the following commutative diagram
\[\begin{tikzcd}
	{H^{\circ}\times^{G^{\circ}}\IV(t)} & {A_0} \\
	{\prod_{i}\Hilb^3(\D_i)} & {\prod_{i}\Hilb^9(\D_i)}
	\arrow["{\Phi_0\circ\iota}", from=1-1, to=1-2]
	\arrow["{\Xi}", hook ,from=1-1, to=2-1]
	\arrow["{\mathrm{Res}}", from=1-2, to=2-2]
	\arrow["{\prod_i\mu_i}", from=2-1, to=2-2]
\end{tikzcd}\]
Here $\mu_i:\Hilb^3(\D_i)\to\Hilb^9(\D_i),Z\mapsto 3Z$, which is unramified at $[3p_i]$ by Lemma \ref{Lemma. Differential of 3pts to 9pts}. The morphism $\Xi$ is an open embedding by Lemma \ref{Lemma. Xi is open immersion}. Hence $\Phi_0\circ\iota$ is unramified at $(1,0)$. This proves the lemma.
\end{proof}

\begin{lem}
\label{Lemma. Differential of 3pts to 9pts}
Let $p\in\IP^1$ be a closed point. For any positive integers $m,n$, the morphism
\[
    \mu:\Hilb^n(\IP^1)\to\Hilb^{mn}(\IP^1),\quad [Z]\mapsto [mZ]
\]
is unramified at the point $[np]$.
\end{lem}
\begin{proof}
    Write the homogeneous coordinates of $\IP^1$ as $[x:y]$ and $p=[1:0]$. Set $r=y/x$ around $p$. Under the isomorphism
    \[
        T_{[np]}\Hilb^n(\IP^1)\cong\Ik[r]/(r^n),
    \]
    a tangent vector $\xi$ at $[np]$ is represented by $r^n+\vep h(r)$ with $\deg h<n$ and $\vep^2=0$. Since $$(r^n+\vep h(r))^m=r^{mn}+\vep mr^{(m-1)n}h(r),$$ $d\mu_{[np]}(\xi)$ vanishes exactly when $h(r)=0$. Thus $d\mu_{[np]}$ is injective, which implies the lemma.
\end{proof}
\begin{lem}
\label{Lemma. Dominance}
    The morphism $\Phi$ is dominant.
\end{lem}
\begin{proof}
     Consider
    \[
        \Tilde{\rho}:\mathrm{PGL}_4\times U\to\IP(V),\quad(g,u)\mapsto g\cdot[F_u],
    \]
    where $F=x_1x_2x_3-tP$. Take $u_*\in U$ corresponding to the cubic equation $F_*=x_1x_2x_3-x_0^3$. Consider the differential map at $(1,u_*)$
  \[
      d\tilde{\rho}_{(1,u_*)}:\mathfrak{pgl}_4\oplus T_{u_*}U\to T_{[F_*]}\IP(V)\cong V/\Ik F_*.
  \]
    For $\xi\in\mathfrak{pgl}_4$, $d\tilde{\rho}_{(1,u_*)}(\xi,0)$ is of the form $\sum_{i=0}^3 l_i\frac{\partial F_*}{\partial x_i}$, where $l_i\in \Ik[x_0,x_1,x_2,x_3]_1$. Thus 
    \begin{equation}
    \label{Equation. diff pgl4 part}
         d\tilde{\rho}_{(1,u_*)}(\mathfrak{pgl}_4\oplus\{0\})=\frac{(x_1x_2,x_2x_3,x_1x_3,x_0^2)\cap\Ik[x_0,x_1,x_2,x_3]_3}{\Ik F_*}. 
    \end{equation}
   By computing partial derivatives of $x_1x_2x_3-tP$ along $P_i$ and $Q_i$, we have
   \begin{equation}
       \label{Equation. diff T}
       d\tilde{\rho}_{(1,u_*)}(\{0\}\oplus T_{u_*}U)\supset \frac{\la x_1^3,x_1^2x_0,x_2^3,x_2^2x_0,x_3^3,x_3^2x_0\ra_{\Ik}+\Ik F_*}{\Ik F_*}.
   \end{equation}
   Moreover, the right-hand sides of (\ref{Equation. diff pgl4 part}) and (\ref{Equation. diff T}) generate $V/\Ik F_*$, thus $d\tilde{\rho}_{(1,u_*)}$ is surjective. Since $\mathrm{PGL}_4\times U$ and $\IP(V)$ are both smooth, $\tilde{\rho}$ is smooth at $(1,u_*)$. In particular, its image contains an open subset of $\IP(V)$. Hence $\tilde{\rho}$ is dominant, implying that $\rho\circ\Phi$ is dominant. Thus
   \[
       19=\dim A\geq\dim\overline{\Image(\Phi)}\geq\dim\overline{\Image(\rho\circ\Phi)}=\dim\IP(V)=19,
   \]
   where the first equality follows from the proof of Lemma \ref{Lemma. Locally quotient}. Since $A$ is irreducible, $\Phi$ is dominant.
\end{proof}
\begin{proof}[Proof of Proposition \ref{Proposition. Etaleness of Phi}]
    It follows from Lemma \ref{Lemma. Locally quotient} that $A$ is irreducible and seminormal, hence integral. Let $\nu:\widetilde{A}\to A$ be the normalization. Fix a point $\widetilde{y}\in\nu^{-1}(y_0)$. Let $R$ be a complete DVR and $\alpha:\Spec R\to\widetilde{A}$ be a morphism mapping the closed point to $\widetilde{y}$ and the generic point into the locus $\nu^{-1}(A^{\mathrm{sm}})$ parametrizing smooth cubic surfaces. Then Lemma \ref{Lemma. Centered normal form} yields a morphism $\beta:\Spec R\to \mathrm{PGL}_4\times^{G}U^{\circ}$ such that 
    \begin{equation}
        \label{Equation. exsit beta}
        \beta(0)=z_0,\quad\Phi\circ\beta=\nu\circ\alpha. 
    \end{equation}
Since $\mathrm{PGL}_4\times^{G}U^{\circ}$ is normal, $\Phi$ lifts to $\widetilde{\Phi}:\mathrm{PGL}_4\times^{G}U^{\circ}\to\widetilde{A}$. The morphisms $\alpha$ and $\widetilde{\Phi}\circ\beta$ agree at the generic point of $\Spec R$ and have the same composition with $\nu$. Since $\nu$ is separated, they agree on  $\Spec R$. Hence
\[
    \widetilde{y}=\alpha(0)=(\widetilde{\Phi}\circ\beta)(0)=\widetilde{\Phi}(z_0).
\]
Therefore $\nu^{-1}(y_0)$ consists of exactly one point. Thus $A$ is geometrically unibranch at $y_0$. By Lemma \ref{Lemma. Dominance} and Lemma \ref{Lemma. Unramified}, $\Phi$ is dominant and unramified at $z_0$. Applying Lemma \ref{Lemma. Etale criterion}, $\Phi$ is \'{e}tale at $z_0$.
\end{proof}
\begin{cor}
    \label{Corollary. Local quotient presentation}
    After replacing $U^{\circ}$ by a smaller $G$-invariant affine neighborhood of $\IV(t)\subset U^{\circ}$, the morphism
    \begin{equation}
    \label{Equation. Local quotient presentation}
      f:  [U^{\circ}/G]\longrightarrow\CY_1^{\cy}
    \end{equation}
    is a local quotient presentation in the sense of Definition \ref{Definition. Local presentation}.
\end{cor}
\begin{proof}
    By Proposition \ref{Proposition. Etaleness of Phi}, we may assume the $\mathrm{PGL}_4$-equivariant morphism $\Phi$ to be  \'{e}tale after shrinking $U^{\circ}$ to a $G$-invariant affine neighborhood of 0. Taking quotient stacks by $\mathrm{PGL}_4$, we obtain an \'etale morphism $[U^{\circ}/G]\to[A/\mathrm{PGL}_4]$. Combining this with Lemma \ref{Lemma. Locally quotient}, the morphism
     \[
         [U^{\circ}/G]\longrightarrow [A/\mathrm{PGL}_4]\cong \CY_1^{\circ}\xhookrightarrow{\quad}\CY_1^{\cy}
     \]
     is \'etale. Its stabilizer map at the origin is an isomorphism by Lemma \ref{Lemma. Automorphism group}. Since $\CY_1^{\cy}$ has affine diagonal, \cite[Proposition 3.2]{alper2020luna} allows a further $G$-invariant affine shrinking of $U^{\circ}$, on which the morphism $f$ is affine. Therefore $f$ is a local quotient presentation. 

     Since every point of $\IV(t)$ specializes to 0 under the one-parameter subgroup $$\IG_m\to(\IG_m)^3\hookrightarrow G,\quad s\mapsto(s^{-1},s^{-1},s^{-1}),$$ any $G$-invariant open subscheme of $U$ containing 0 also contains $\IV(t)$. This finishes the proof.
\end{proof}
\section{Local VGIT and wall crossing}
\label{Section. sec 5}
\subsection{VGIT chambers and quotients}
Write $G=T\rtimes S_3$, where $T=(\IG_m)^3$. Let $e_1,e_2,e_3\in X^*(T)$ be characters defined by 
 \[
     e_i:T\longrightarrow\IG_m,\quad (\lambda_1,\lambda_2,\lambda_3)\mapsto\lambda_i.
 \]
For the $G$-action on $R:=\Ik[U]=\Ik[t,P_1,Q_1,P_2,Q_2,P_3,Q_3]$, the $T$-weights are
    \[
        \wt(t)=e_1+e_2+e_3,\quad\wt(P_i)=-2e_i,\quad\wt(Q_i)=-3e_i.
    \]
    The $S_3$-action fixes $t$ and permutes the three pairs $(P_i,Q_i)$. Since $e_1+e_2+e_3$ is $S_3$-invariant, it extends to a character $\chi\in X^*(G)$
\[
    \chi:G\longrightarrow\IG_m,\quad ((\lambda_1,\lambda_2,\lambda_3),\sigma)\mapsto \lambda_1\lambda_2\lambda_3.
\]
Define $\theta:=6\chi\in X^*(G)$.
\begin{lem}
\label{Lemma. VGIT chambers}
    The VGIT chambers of $U$ with respect to $\theta$ are the following:
    \[
        U^+_{\theta}=\bigcap_i\{(P_i,Q_i)\neq(0,0)\},\quad U^-_{\theta}=\{t\neq0\}.
    \]
\end{lem}
\begin{proof}
    Consider the $G$-action on $U$ and define $R_n$ as in (\ref{Equation. def of An}). A monomial $M=t^q\prod_i P_i^{k_i}Q_i^{l_i}\in\Ik[U]$ has weight  
    \begin{equation}
    \label{Equation. weight of monomial}
    \wt(M)=\sum_{i}(q-2k_i-3l_i)e_i.
    \end{equation}
    If $M$ occurs in an element of $R_n$ for some $n<0$, then $\wt(M)=-n\theta$, which implies that $q-2k_i-3l_i=-6n$ for each $i$. In particular, $q>0$. Thus every element of the VGIT ideal $I_{\theta}^-$ vanishes on $\IV(t)$. Note that $t^6\in R_{-1}$, hence $\IV(I_{\theta}^-)=\IV(t)$. Therefore $U_{\theta}^-=U\backslash V(I_{\theta}^-)=\{t\neq0\}$. 
    
    If $M$ occurs in an element of $R_m$ for some $m>0$, then $q-2k_i-3l_i=-6m$ for each $i$. Thus $2k_i+3l_i=q+6m>0$, which implies that  $(k_i,l_i)\neq(0,0)$ for each $i$. Hence each element of the VGIT ideal $I_{\theta}^+$ vanishes on $\bigcup_i\{P_i=Q_i=0\}\subset U$. This proves
    \begin{equation}
    \label{Equation. inclusion of negaive chamber}
        U_{\theta}^+\subset\bigcap_i\{(P_i,Q_i)\neq(0,0)\}.
    \end{equation}
    Conversely, suppose $u_0\in U$ with all three pairs $(P_i(u_0),Q_i(u_0))$  nonzero. Choose $c\in\Ik$ such that $P_i(u_0)^3+cQ_i(u_0)^2\neq0$ for each $i$. Then the polynomial $F_c=\prod_i(P_i^3+cQ_i^2)$ is $S_3$-invariant and has weight $-6(e_1+e_2+e_3)=-\theta$. Thus $F_c\in R_1$ and $F_c(u_0)\neq0$. Therefore the reverse inclusion of (\ref{Equation. inclusion of negaive chamber}) holds.
\end{proof}
Let $\pi:U\to U\git G$ and $\pi_{\pm}:U_{\theta}^{\pm}\to U_{\theta}^{\pm}\git G$ denote the good quotients. The inclusions induce the following commutative diagram:
\[\begin{tikzcd}
	{U_{\theta}^-} & {U} & {U_{\theta}^+} \\
	{U_{\theta}^-\git G} & {U\git G} & {U_{\theta}^+\git G}
	\arrow[hook, from=1-1, to=1-2]
	\arrow["{\pi_-}"', from=1-1, to=2-1]
	\arrow["\pi", from=1-2, to=2-2]
	\arrow[hook', from=1-3, to=1-2]
	\arrow["{\pi_+}", from=1-3, to=2-3]
	\arrow["{p_-}", from=2-1, to=2-2]
	\arrow["{p_+}"', from=2-3, to=2-2]
\end{tikzcd}\]
\begin{lem}
    \label{Lemma. morphism of quotient of VGIT}
    Set $v=\pi(0)$. Then 
    \begin{enumerate}
        \item $p_-$ is an isomorphism;
        \item The reduced structure of $p_+^{-1}(v)$ is isomorphic to $\IP^3$.
    \end{enumerate}
\end{lem}
\begin{proof}
    A monomial $M=t^q\prod_iP_i^{k_i}Q_i^{l_i}$ occurring in a $G$-invariant polynomial satisfies $q-2k_i-3l_i=0$ for each $i$. Since $k_i,l_i\geq 0$, one has $q\geq 0$. Thus $R^{G}=(R[t^{-1}])^{G}$. Hence $p_-$ is an isomorphism. 

    Set $Z:=\IV(t)^+_{\theta}=U^+_{\theta}\cap \IV(t)$. We claim that $Z=\pi^{-1}(v)\cap U_{\theta}^+$ set-theoretically. Indeed, for a non-constant monomial $M=t^q\prod_iP_i^{k_i}Q_i^{l_i}$ occurring in a $G$-invariant polynomial, we have $q>0$. Thus $M$ vanishes on $\IV(t)$. The natural homomorphism $R^{G}\to R/(t)$ factors through $R^{G}\to\Ik,f\mapsto f(0)$. Thus 
    \begin{equation}
    \label{Equation. pi(V(t))}
    \pi(\IV(t))=\pi(0)=v,
    \end{equation}
 which implies $Z\subset \pi^{-1}(v)\cap U_{\theta}^+$. For the reverse inclusion, suppose $u\in U_{\theta}^+$ and $t(u)\neq 0$. Consider
 \begin{equation}
 \label{Equation. alg indep poly}
     P_u(X):=\prod_i(P_i(u)^3X+Q_i(u)^2)=c_0(u)X^3+c_1(u)X^2+c_2(u)X+c_3(u).
 \end{equation}
Because $u\in U_{\theta}^+$, $P_u(X)$ is nonzero. Thus $c_j(u)\neq 0$ for some $j$, which implies that $(t^6c_j)(u)\neq 0$. Since $t^6c_j$ is $G$-invariant and vanishes at the origin, we have $\pi(u)\neq v$. This proves the reverse inclusion. 

Since $\pi|_{U_{\theta}^+}=p_+\circ\pi_+$, we have $Z=\pi_+^{-1}(p_+^{-1}(v))$. Thus $p_+^{-1}(v)=\pi_+(Z)$ set-theoretically by the surjectivity of $\pi_+$. Note that the closed immersion $Z\git G\hookrightarrow U_{\theta}^+\git G$ has underlying image precisely $\pi_+(Z)$. Hence $Z\git G=p_+^{-1}(v)$ set-theoretically.

Set $\oR:=R/(t)=\Ik[\IV(t)]=\Ik[P_1,Q_1,P_2,Q_2,P_3,Q_3]$. As $\IV(t)$ is preserved by the $G$-action, we can consider the $G$-action on $\IV(t)$ and define $\oR_n$ as in (\ref{Equation. def of An}). A monomial $\prod_iP_i^{k_i}Q_i^{l_i}$ occurs in an element of $\oR_n$ precisely when $2k_i+3l_i=6n$ for each $i$. Hence $\oR_n=(W_n^{\otimes 3})^{S_3}=\mathrm{Sym}^3(W_n)$, where
\[
    W_n=\la P^{3n},P^{3(n-1)}Q^2,\cdots,Q^{2n}\ra=\mathrm{Sym}^n\la P^3,Q^2\ra.
\]
Then $\dim \oR_n=\binom{n+3}{3}$. Consider the morphism induced by (\ref{Equation. alg indep poly}):
\[
    (\IA^2)^3\longrightarrow \IA^4,\quad (P_i,Q_i)_i\mapsto(c_0,c_1,c_2,c_3).
\]
Since $\Ik$ is algebraically closed, the morphism is surjective. Thus $c_0,c_1,c_2,c_3$ are algebraically independent. Moreover, $c_i\in \oR_{1}$ for $0\leq i\leq3$. As the degree $n$ part of the subalgebra generated by $\{c_i\}_{i=0}^3$ has dimension $\binom{n+3}{n}$, we obtain that
\[
    \bigoplus_{n\geq 0}\oR_n=\Ik[c_0,c_1,c_2,c_3],\quad\deg c_i=1.
\]
Taking Proj gives $Z\git G\cong\Proj\Ik[c_0,c_1,c_2,c_3]\cong\IP^3$. 
\end{proof}
\begin{rmk}
    The scheme-theoretic fiber of $p_+$ at $v$ is nonreduced. In fact, it has generic multiplicity 2 along the reduced support. Consider the affine chart of $U_{\theta}^+\git G$ associated to $c_0c_3=\prod_iP_i^3Q_i^2$, and set
    \[
        \Omega=\bigg\{\prod_iP_iQ_i\neq 0\bigg\}\subset U_{\theta}^+,\quad \tau=t\prod_i\frac{Q_i}{P_i},\quad a_j=\frac{c_j}{c_0}\quad(0\leq j\leq3).
    \]
    Then one can verify that the coordinate ring of $\Omega\git G$ is $B:=\Ik[\Omega]^G=\Ik[\tau,a_1,a_2,a_3,a_3^{-1}]$ and the reduced exceptional divisor on this chart is $\IV(\tau)$. Let $\fm_v\subset R^G$ be the maximal ideal of $v$. For every nonconstant monomial $t^q\prod_iP_i^{k_i}Q_i^{l_i}$ occurring in an element of $\fm_v$, (\ref{Equation. weight of monomial}) implies that $q\geq 2$. Thus $\fm_vB\subset (\tau^2)$. Conversely, $\tau^2/a_3=t^2P_1P_2P_3\in\fm_v$ and $a_3$ is a unit in $B$. Hence $\fm_vB=(\tau^2)$. Therefore the scheme-theoretic fiber on this chart is $\Spec(B/(\tau^2))$, which has multiplicity 2 along the reduced support.
\end{rmk}
\subsection{The modular interpretation of chambers}
\begin{lem}
    \label{Lemma. shrinking U A1A2}
    There exists a $G$-invariant open neighborhood $\IV(t)\subset W\subset U$ satisfying the following:
    \begin{enumerate}
        \item For any $u\in W\cap U_{\theta}^-$, $\CX_u$ is GIT-semistable;
        \item For any $u\in W\cap U_{\theta}^-\cap U_{\theta}^+$, $\CX_u$ is GIT-stable.
    \end{enumerate}
\end{lem}
\begin{proof}
     From Lemma \ref{Lemma. morphism of quotient of VGIT}, $U_{\theta}^-\git G\cong U\git G\cong \Spec R^G$. Now consider $u_*\in U_{\theta}^-$ corresponding to the cubic surface $X_*=\{x_1x_2x_3=x_0^3\}\subset \IP^3$. The orbit 
     $$\CO_*:=G\cdot u_*=\{t\neq0,P_i=Q_i=0\}$$
     is closed in $U_{\theta}^-$. Since every $G$-invariant function takes the same value at $u_*$ and 0, $\pi_-(\CO_*)=v$. Let $(U_{\theta}^-)^{\mathrm{ss}}\subset U_{\theta}^-$ be the open subscheme parametrizing GIT-semistable cubic surfaces. It is $G$-invariant and contains $\CO_*$. Set $Z:=U_{\theta}^-\backslash (U_{\theta}^-)^{\mathrm{ss}}$. Then $Z$ is a closed $G$-invariant subscheme of $U_{\theta}^-$. Since $\pi_-$ is an affine good quotient, $\pi_-(Z)$ is closed in $U_{\theta}^-\git G$. Hence  $v\notin\pi_-(Z)$. Indeed, if $\pi_-(z)=v$ for some $z\in Z$, then $\overline{G\cdot z}\subset Z$. Since $\CO_*$ is a closed orbit in $U_{\theta}^-$, by \cite{Mumford_GIT} we have $\CO_*\subset \overline{G\cdot z}\subset Z$, contradicting the fact that $\CO_*\subset (U_{\theta}^-)^{\mathrm{ss}}$. 
     
     As $\pi_-(Z)$ is closed and does not contain $v$, there exists $\bar{h}\in R^{G}$ such that $\bar{h}|_{\pi_-(Z)}=0$ and $\bar{h}(v)\neq0$. Set 
     $$W=\pi^{-1}(\Spec(R^{G})_{\bar{h}})=\Spec R_{\pi^*\bar{h}}\subset U.$$ 
     From (\ref{Equation. pi(V(t))}), we have $\IV(t)\subset W$. If $u\in W_{\theta}^-:=W\cap U_{\theta}^-$ corresponds to a GIT-unstable cubic surface, then $\pi_-(u)\in \pi_-(Z)$. This is impossible as $\bar{h}|_{\pi_-(Z)}=0$. Hence every point in $W_{\theta}^-$ corresponds to a GIT-semistable cubic surface.

     Let $\gamma:W_{\theta}^-\to\IP(V)^{\mathrm{ss}}$ be the natural morphism defined by the family $\CX\to U$ and $q:\IP(V)^{\mathrm{ss}}\to M^{\mathrm{GIT}}$ be the GIT quotient map, where $V:=H^0(\IP^3,\CO_{\IP^3}(3))$ and $M^{\mathrm{GIT}}$ is the GIT moduli space of cubic surfaces. Since $\gamma$ is $G$-equivariant, the composition $q\circ\gamma$ is $G$-invariant. Thus it factors through $\pi|_{W_{\theta}^-}:W_{\theta}^-\to W_{\theta}^-\git G$. Denote the morphism $W_{\theta}^-\git G\to M^{\mathrm{GIT}}$ by $\overline{\gamma}$. Set $m_*=[X_*]\in M^{\mathrm{GIT}}$. We claim that after replacing $W$ by a $G$-invariant open subscheme, $\overline{\gamma}^{-1}(m_*)=v$. 
     
     Set $H_*=\Aut_{\IP^3}(X_*)\cong\{(\lambda_i)\in(\IG_m)^3|\lambda_1\lambda_2\lambda_3=1\}\rtimes S_3$ and  $S=\IV(t-1)\cap W$. Consider the natural morphism
     \[
         \mathrm{PGL}_4\times^{H_*}S\longrightarrow \IP(V)^{\mathrm{ss}}.
     \]
      By a direct computation, its differential map at $(1,u_*)$ is an isomorphism. Thus $$[S/H_*]\to[\IP(V)^{\mathrm{ss}}/\mathrm{PGL}_4]$$ is strongly \'{e}tale at $u_*.$ The natural isomorphism $W_{\theta}^-\cong G\times^{H_*}S$ induces the equivalence $[W_{\theta}^-/G]\cong[S/H_*]$, hence $[W_{\theta}^-/G]\to[\IP(V)^{\mathrm{ss}}/\mathrm{PGL}_4]$ is strongly \'{e}tale at $u_*$. By \cite[Proposition 2.7]{AFS17}, $\overline{\gamma}:W_{\theta}^-\git G\to M^{\mathrm{GIT}}$ is \'{e}tale at $v$. In particular, $\overline{\gamma}$ is quasi-finite at $v$, and thus we can shrink $W$ to a $G$-invariant open neighborhood of $\IV(t)$ such that $\overline{\gamma}^{-1}(m_*)=v$. Hence the claim is proved. 

      Now we prove that $\pi_-^{-1}(v)=U_{\theta}^-\cap\bigcup_i\IV(P_i,Q_i)$. Indeed, it follows from (\ref{Equation. weight of monomial}) that a nonconstant monomial occurring in a $G$-invariant polynomial in $R[t^{-1}]$ vanishes on $\bigcup_i\IV(P_i,Q_i)$. Thus, as in the proof of (\ref{Equation. pi(V(t))}), every invariant function takes the same value on this locus as at the origin, which gives $$\pi_-^{-1}(v)\supseteq U_{\theta}^-\cap\bigcup_i\IV(P_i,Q_i).$$ If $u\in U_{\theta}^-\backslash\bigcup_i\IV(P_i,Q_i)$, then $(P_i(u),Q_i(u))\neq(0,0)$ for each $i$. So there exist $a,b\in\Ik$ such that the $G$-invariant polynomial $F_{a,b}=t^6\prod_i(aP_i^3+bQ_i^2)$ is nonzero at $u$, implying the reverse inclusion. As the image of every strictly semistable cubic surface under $q$ is $m_*$, the locus in $W_{\theta}^-$ parametrizing strictly semistable cubic surfaces is
      \[
          (\gamma)^{-1}(q^{-1}(m_*))=(\pi|_{W_{\theta}^-})^{-1}(\overline{\gamma}^{-1}(m_*))=(\pi|_{W_{\theta}^-})^{-1}(v)=W_{\theta}^-\cap\bigcup_i\IV(P_i,Q_i).
      \]
     Since $\pi(\IV(t))=\pi(u_*)$, we have $\IV(t)\subset W$. Hence the $G$-invariant open neighborhood $\IV(t)\subset W\subset U$ satisfying the requirement.
\end{proof}

\begin{prop}
    \label{Proposition. Local VGIT Cartesian}
    There exists a $G$-invariant affine open neighborhood $\IV(t)\subset W\subset U^{\circ}$, such that $W^{\pm}_{\theta}:=W\cap U^{\pm}_{\theta}$ fit into the 
    following Cartesian diagram 
    \begin{equation}
    \label{Equation. VGIT diagram}
    \begin{tikzcd}
	{[W^-_{\theta}/G]} & {[W/G]} & {[W^+_{\theta}/G]} \\
	{\CY^{\mathrm{K}}} & {\CY_1^{\cy}} & {\CY^{\ksba}}
	\arrow[from=1-1, to=1-2]
	\arrow[from=1-1, to=2-1]
	\arrow[from=1-2, to=2-2]
	\arrow[from=1-3, to=1-2]
	\arrow[from=1-3, to=2-3]
	\arrow[hook, from=2-1, to=2-2]
	\arrow[hook', from=2-3, to=2-2]
\end{tikzcd}
\end{equation}
\end{prop}
\begin{proof}
    Let $\IV(t)\subset W\subset U$ be a $G$-invariant open neighborhood satisfying Lemma \ref{Lemma. shrinking U A1A2} and replace $W$ by $W\cap U^{\circ}$. Set $\iota:[W/G]\to\CY_1^{\cy}$. For $u\in W$, the construction of $W$ gives that  $\iota(u)\in\CY^{\mathrm{K}}$ if and only if $u\in W\cap U_{\theta}^-$. So the left square in (\ref{Equation. VGIT diagram}) is Cartesian. 
    
    For $u\in \IV(t)$, \cite[Theorem 1.5]{GKS21} implies that $\iota(u)\in\CY^{\ksba}$ precisely when $u\in W_{\theta}^+$. Now let $u\in W_{\theta}^-$. If $u\in W^+_{\theta}$, then the underlying surface of $\iota(u)$ is GIT-stable, thus has at worst $A_1$-singularities. If $u\notin W^+_{\theta}$, then we may assume $P_1(u)=Q_1(u)=0$. On the affine chart $x_1=1$, near $[0:1:0:0]$ the equation is $x_2x_3=t(x_0^3+P_2x_0x_2^2+Q_2x_2^3+P_3x_0x_3^2+Q_3x_3^3)$. Assign weights $\wt(x_2)=\wt(x_3)=\frac12$ and $\wt(x_0)=\frac13$. Then the equation is semiquasihomogeneous with principal part $x_2x_3-tx_0^3$, which is an $A_2$-singularity since $t\neq0$. Hence the underlying surface of $\iota(u)$ has an $A_2$-singularity by \cite[Lemma 1 and Corollary]{BW79}. Thus $\iota(u)\in\CY^{\ksba}$ precisely when $u\in W^+_{\theta}$ by \cite[Theorem 1.5]{GKS21}. Therefore the right square in (\ref{Equation. VGIT diagram}) is Cartesian.
\end{proof}
\subsection{Wall crossing morphisms}
Return to the commutative diagram (\ref{Equation. wall crossing diagram})
\[
    \begin{tikzcd}
	{\CY^{\mathrm{K}}} & {\CY_1^{\cy}} & {\CY^{\ksba}} \\
	{Y^{\mathrm{K}}} & {Y_1^{\cy}} & {Y^{\ksba}}
	\arrow[hook, from=1-1, to=1-2]
	\arrow["{\phi_{\mathrm{K}}}", from=1-1, to=2-1]
	\arrow["{\phi_{\mathrm{\cy}}}", from=1-2, to=2-2]
	\arrow[hook', from=1-3, to=1-2]
	\arrow["{\phi_{\mathrm{\ksba}}}", from=1-3, to=2-3]
	\arrow["{\pi_{\mathrm{K}}}", from=2-1, to=2-2]
	\arrow["{\pi_{\ksba}}"', from=2-3, to=2-2]
\end{tikzcd}
\]
Now we can describe the two morphisms $\pi_{\mathrm{K}}$ and $\pi_{\ksba}$.
\begin{cor}
\label{Corollary. describe wall crossing morphisms}
   With the notation above, the following hold.
 \begin{enumerate}
     \item $\pi_{\mathrm{K}}$ is an isomorphism. In particular, $Y_1^{\cy}$ is isomorphic to the GIT moduli space of cubic surfaces.
     \item $\pi_{\ksba}$ is an isomorphism outside $q_0:=\phi_{\cy}(p_0)$ and $\pi_{\ksba}$ contracts an exceptional divisor $E\cong\IP^3$ to $q_0$.
 \end{enumerate}
\end{cor}
\begin{proof}
     Set $V:=Y_1^{\cy}\backslash\{q_0\}$ and $\CV:=\phi_{\cy}^{-1}(V)$. We first claim that
     \[
         \CV\subset\CY^{\mathrm{K}}\cap\CY^{\ksba}.
     \]
     For $x\in \CV$, let $x_0$ be the unique closed point in the fiber of $\phi_{\cy}$ containing $x$. As $\phi_{\cy}(x)\neq q_0$, we have $x_0\neq p_0$. By Proposition \ref{Proposition. Unique closed pt}, $x_0\in\CY^{\mathrm{K}}\cap\CY^{\ksba}$.  Since $\CY^{\mathrm{K}}$ and $\CY^{\ksba}$ are both open substacks of $\CY_1^{\cy}$ and $x_0$ is a specialization of $x$, we have $x\in\CY^{\mathrm{K}}\cap\CY^{\ksba}$, implying the claim. Now the restrictions 
     \[
     \phi_{\cy}|_{\CV}:\CV\longrightarrow V,\quad\phi_{\mathrm{K}}|_{\CV}:\CV\longrightarrow\pi_{\mathrm{K}}^{-1}(V),\quad\phi_{\ksba}|_{\CV}:\CV\longrightarrow\pi_{\ksba}^{-1}(V)
     \]
     are all good moduli space morphisms. By the uniqueness of good moduli spaces, 
     \[
         \pi_{\mathrm{K}}^{-1}(V)\xrightarrow{\sim}V,\quad \pi_{\ksba}^{-1}(V)\xrightarrow{\sim}V.
     \]
    As $f:[W/G]\to\CY_1^{\cy}$ is strongly \'{e}tale at 0, after replacing $W$ by a $G$-invariant open neighborhood of $\IV(t)$, we may assume that $f$ is strongly \'{e}tale by \cite[Proposition 2.7(3)]{AFS17}. Then $W\git G\to Y_1^{\cy}$ is \'{e}tale. Consider the following diagram
\[\begin{tikzcd}
	{[W^-_{\theta}/G]} & {[W/G]} & {W\git G} \\
	{\CY^{\mathrm{K}}} & {\CY_1^{\cy}} & {Y_1^{\cy}}
	\arrow[from=1-1, to=1-2]
	\arrow[from=1-1, to=2-1]
	\arrow[from=1-2, to=1-3]
	\arrow[from=1-2, to=2-2]
	\arrow[from=1-3, to=2-3]
	\arrow[from=2-1, to=2-2]
	\arrow[from=2-2, to=2-3]
\end{tikzcd}\]
The left and right squares are Cartesian by Proposition \ref{Proposition. Local VGIT Cartesian} and \cite[Proposition 2.7(2)]{AFS17}, respectively. Then we obtain 
\begin{equation}
    \label{Equation. Cartesian in cor}
    [W_{\theta}^-/G]\cong W\git G\times_{Y_1^{\cy}}\CY^{\mathrm{K}}.
\end{equation}
    Pulling back the good moduli space morphism $\CY^{\mathrm{K}}\to Y^{\mathrm{K}}$ along $W\git G\times_{Y_1^{\cy}}Y^{\mathrm{K}}\to Y^{\mathrm{K}},$ by \cite[Proposition 4.7]{Alper13}, we obtain a good moduli space morphism
    \begin{equation}
        \label{Equation. cor good moduli mor}
        W\git G\times_{Y_1^{\cy}}\CY^{\mathrm{K}}\longrightarrow W\git G\times_{Y_1^{\cy}}Y^{\mathrm{K}}.
    \end{equation}
    Combining (\ref{Equation. Cartesian in cor}), (\ref{Equation. cor good moduli mor}) and the uniqueness of good moduli spaces, we have 
    \begin{equation}
    \label{Equation. local Cartesian K}
    W_{\theta}^-\git G\cong W\git G\times_{Y_1^{\cy}}Y^{\mathrm{K}}.
    \end{equation}
Hence $\pi_K$ is an isomorphism near $q_0$ by Lemma \ref{Lemma. morphism of quotient of VGIT} (1) and \cite[\href{https://stacks.math.columbia.edu/tag/041Y}{Tag 041Y}]{stacks-project}. Therefore $\pi_{\mathrm{K}}$ is an isomorphism. 

The same proof as (\ref{Equation. local Cartesian K}) shows that
\begin{equation*}
    W_{\theta}^+\git G\cong W\git G\times_{Y_1^{\cy}}Y^{\ksba}.
\end{equation*}
Thus Lemma \ref{Lemma. morphism of quotient of VGIT} (2) implies that $(\pi_{\ksba}^{-1}(q_0))_{\mathrm{red}}\cong\IP^3$. This completes the proof.
\end{proof}
\begin{rmk}
    The pairs parametrized by $E\cong\IP^3$ admit the following description. Given $c=[c_0:c_1:c_2:c_3]\in\IP^3$, allowing constant factors, factor the corresponding polynomial as 
    \[
        c_0X^3+c_1X^2+c_2X+c_3=\prod_i(A_iX+B_i).
    \]
    Choose $P_i,Q_i\in\Ik$ such that $P_i^3=A_i,Q_i^2=B_i$, and write
        $r^3+P_ir+Q_i=\prod_{\alpha=1}^3(r-\lambda_{i,\alpha}).$
    On the surface $X_0=(x_1x_2x_3=0)\subset\IP^3$, consider the line 
    \[
        L_{ij}^{\alpha\beta}:=\IV(x_k,x_0-\lambda_{i,\alpha}x_i-\lambda_{j,\beta}x_j),\quad\{i,j,k\}=\{1,2,3\}.
    \]
    Then $c\in\IP^3$ represents the boundary polarized CY pair $(X_0,\frac19D_c)$, where $D_c=\sum_{1\leq i<j\leq3}\sum_{\alpha,\beta=1}^3L_{ij}^{\alpha\beta}$.

    Up to isomorphism, the pair $(X_0,\frac19D_c)$ is independent of the choice of $P_i$ and $Q_i$. Indeed, any other choices differ by
    \[
        (P_i,Q_i)\longmapsto (\mu_i^2P_i,\mu_i^3Q_i)
    \]
    for some $\mu_i\in\Ik^{\times}$. This rescales the roots $\lambda_{i,\alpha}$ by $\mu_i$, and the resulting change of boundary divisor is induced by the automorphism
    \[
        [x_0:x_1:x_2:x_3]\longmapsto[x_0:\mu_1^{-1}x_1:\mu_2^{-1}x_2:\mu_3^{-1}x_3]
    \]
    of $X_0$.
\end{rmk}

\bibliographystyle{alpha}
\bibliography{ref}

\newcommand{\etalchar}[1]{$^{#1}$}
\begin{thebibliography}{ABB{\etalchar{+}}23}

\bibitem[ABB{\etalchar{+}}23]{ABB23}
Kenneth Ascher, Dori Bejleri, Harold Blum, Kristin DeVleming, Giovanni
  Inchiostro, Yuchen Liu, and Xiaowei Wang.
\newblock Moduli of boundary polarized {Calabi-Yau} pairs, 2023.
\newblock \href{https://arxiv.org/abs/2307.06522}{\textsf{arXiv:2307.06522}}.

\bibitem[ADL24]{ADL24}
Kenneth Ascher, Kristin DeVleming, and Yuchen Liu.
\newblock Wall crossing for {K}-moduli spaces of plane curves.
\newblock {\em Proc. Lond. Math. Soc. (3)}, 128(6):Paper No. e12615, 113, 2024.

\bibitem[AFS17]{AFS17}
Jarod Alper, Maksym Fedorchuk, and David~Ishii Smyth.
\newblock Second flip in the {H}assett-{K}eel program: existence of good moduli
  spaces.
\newblock {\em Compos. Math.}, 153(8):1584--1609, 2017.

\bibitem[AHLH23]{AHLH23}
Jarod Alper, Daniel Halpern-Leistner, and Jochen Heinloth.
\newblock Existence of moduli spaces for algebraic stacks.
\newblock {\em Invent. Math.}, 234(3):949--1038, 2023.

\bibitem[AHR20]{alper2020luna}
Jarod Alper, Jack Hall, and David Rydh.
\newblock A {L}una \'etale slice theorem for algebraic stacks.
\newblock {\em Ann. of Math. (2)}, 191(3):675--738, 2020.

\bibitem[Alp13]{Alper13}
Jarod Alper.
\newblock Good moduli spaces for {A}rtin stacks.
\newblock {\em Ann. Inst. Fourier (Grenoble)}, 63(6):2349--2402, 2013.

\bibitem[AOV08]{AOV08}
Dan Abramovich, Martin Olsson, and Angelo Vistoli.
\newblock Tame stacks in positive characteristic.
\newblock {\em Ann. Inst. Fourier (Grenoble)}, 58(4):1057--1091, 2008.

\bibitem[BL24]{BL24}
Harold Blum and Yuchen Liu.
\newblock Good moduli spaces for boundary polarized {Calabi-Yau} surface pairs,
  2024.
\newblock \href{https://arxiv.org/abs/2407.00850}{\textsf{arXiv:2407.00850}}.

\bibitem[BLXZ25]{BLXZ25}
Harold Blum, Yuchen Liu, Chenyang Xu, and Ziquan Zhuang.
\newblock Relative stability theory and properness of {K}-moduli spaces, 2025.
\newblock \href{https://arxiv.org/abs/2510.06197}{\textsf{arXiv:2510.06197}}.

\bibitem[BW79]{BW79}
J.~W. Bruce and C.~T.~C. Wall.
\newblock On the classification of cubic surfaces.
\newblock {\em J. London Math. Soc. (2)}, 19(2):245--256, 1979.

\bibitem[FSW25]{FSW25}
Hanlong Fang, Luca Schaffler, and Xian Wu.
\newblock Fineness and smoothness of a {KSBA} moduli of marked cubic surfaces.
\newblock {\em Proc. Amer. Math. Soc.}, 153(12):5133--5146, 2025.

\bibitem[Fuj90]{Fujita90}
Takao Fujita.
\newblock On singular del {P}ezzo varieties.
\newblock In {\em Algebraic geometry ({L}'{A}quila, 1988)}, volume 1417 of {\em
  Lecture Notes in Math.}, pages 117--128. Springer, Berlin, 1990.

\bibitem[GKS21]{GKS21}
Patricio Gallardo, Matt Kerr, and Luca Schaffler.
\newblock Geometric interpretation of toroidal compactifications of moduli of
  points in the line and cubic surfaces.
\newblock {\em Adv. Math.}, 381:Paper No. 107632, 48, 2021.

\bibitem[Kol13]{Kollar13}
J\'anos Koll\'ar.
\newblock {\em Singularities of the minimal model program}, volume 200 of {\em
  Cambridge Tracts in Mathematics}.
\newblock Cambridge University Press, Cambridge, 2013.
\newblock With a collaboration of S\'andor Kov\'acs.

\bibitem[Kol16]{Kollar16}
J\'anos Koll\'ar.
\newblock Sources of log canonical centers.
\newblock In {\em Minimal models and extremal rays ({K}yoto, 2011)}, volume~70
  of {\em Adv. Stud. Pure Math.}, pages 29--48. Math. Soc. Japan, [Tokyo],
  2016.

\bibitem[Kol23]{Kollar23}
J\'anos Koll\'ar.
\newblock {\em Families of varieties of general type}, volume 231 of {\em
  Cambridge Tracts in Mathematics}.
\newblock Cambridge University Press, Cambridge, 2023.
\newblock With the collaboration of Klaus Altmann and S\'andor J. Kov\'acs.

\bibitem[LWX21]{LWX21}
Chi Li, Xiaowei Wang, and Chenyang Xu.
\newblock Algebraicity of the metric tangent cones and equivariant
  {K}-stability.
\newblock {\em J. Amer. Math. Soc.}, 34(4):1175--1214, 2021.

\bibitem[LXZ22]{LXZ22}
Yuchen Liu, Chenyang Xu, and Ziquan Zhuang.
\newblock Finite generation for valuations computing stability thresholds and
  applications to {K}-stability.
\newblock {\em Ann. of Math. (2)}, 196(2):507--566, 2022.

\bibitem[MFK94]{Mumford_GIT}
D.~Mumford, J.~Fogarty, and F.~Kirwan.
\newblock {\em Geometric invariant theory}, volume~34 of {\em Ergebnisse der
  Mathematik und ihrer Grenzgebiete (2) [Results in Mathematics and Related
  Areas (2)]}.
\newblock Springer-Verlag, Berlin, third edition, 1994.

\bibitem[Nar82]{Naruki82}
Isao Naruki.
\newblock Cross ratio variety as a moduli space of cubic surfaces.
\newblock {\em Proc. London Math. Soc. (3)}, 45(1):1--30, 1982.
\newblock With an appendix by Eduard Looijenga.

\bibitem[Oda13]{Odaka13}
Yuji Odaka.
\newblock The {GIT} stability of polarized varieties via discrepancy.
\newblock {\em Ann. of Math. (2)}, 177(2):645--661, 2013.

\bibitem[OSS16]{OSS16}
Yuji Odaka, Cristiano Spotti, and Song Sun.
\newblock Compact moduli spaces of del {P}ezzo surfaces and
  {K}\"ahler-{E}instein metrics.
\newblock {\em J. Differential Geom.}, 102(1):127--172, 2016.

\bibitem[Rei94]{Reid94}
Miles Reid.
\newblock Nonnormal del {P}ezzo surfaces.
\newblock {\em Publ. Res. Inst. Math. Sci.}, 30(5):695--727, 1994.

\bibitem[Sch24]{Schock24}
Nolan Schock.
\newblock Moduli of weighted stable marked cubic surfaces, 2024.
\newblock \href{https://arxiv.org/abs/2305.06922}{\textsf{arXiv:2305.06922}}.

\bibitem[{Sta}26]{stacks-project}
The {Stacks project authors}.
\newblock The stacks project.
\newblock \url{https://stacks.math.columbia.edu}, 2026.

\bibitem[Xu25]{Xubook}
Chenyang Xu.
\newblock {\em K-stability of {F}ano varieties}, volume~50 of {\em New
  Mathematical Monographs}.
\newblock Cambridge University Press, Cambridge, 2025.

\bibitem[Zha24]{Zhao24}
Junyan Zhao.
\newblock Compactifications of moduli of del {P}ezzo surfaces via line
  arrangement and {K}-stability.
\newblock {\em Canad. J. Math.}, 76(6):2115--2135, 2024.

\end{thebibliography}

\end{document}